\documentclass[
	submission,
copyright,creativecommons,UKenglish]{eptcs}
\providecommand{\event}{QPL 2020}

\usepackage{babel}
\usepackage{amsmath,amssymb}
\usepackage{xspace}
\usepackage{amsthm}
\usepackage{physics}
\usepackage{mathtools}
\usepackage{latexsym}
\usepackage{stmaryrd}
\usepackage{microtype}
\usepackage{picins}

\usepackage{xparse}
\usepackage{environ} 
\usepackage[noadjust]{cite}

\DeclareFontFamily{U}{mathux}{\hyphenchar\font45}
\DeclareFontShape{U}{mathux}{m}{n}{
      <5> <6> <7> <8> <9> <10>
      <10.95> <12> <14.4> <17.28> <20.74> <24.88>
      mathux10
      }{}
\DeclareSymbolFont{mathux}{U}{mathux}{m}{n}
\DeclareFontSubstitution{U}{mathux}{m}{n}
\DeclareMathSymbol{\bigovee}{1}{mathux}{"8F}

\DeclareSymbolFont{cmlargesymbols}{OMX}{cmex}{m}{n}
\DeclareMathSymbol{\cmcoprod}{\mathop}{cmlargesymbols}{"60}
\let\coprod\cmcoprod

\usepackage{xcolor}
\usepackage{tikz}
\usepackage{tikz-cd}
\usetikzlibrary{arrows, matrix}
\usepackage[textsize=footnotesize]{todonotes}

\usepackage{enumitem}
\newlist{myenum}{enumerate}{2}
\setlist[myenum,1]{label=\textup{(\roman*)},ref=(\roman*)}
\setlist[myenum,2]{label=\textup{(\alph*)},ref=(\alph*)}

\newcommand{\urlalt}[2]{\href{#1}{\urlstyle{rm}\nolinkurl{#2}}}

\newcounter{main}
\newtheorem{theorem}[main]{Theorem}
\newtheorem{proposition}[main]{Proposition}
\newtheorem{lemma}[main]{Lemma}
\newtheorem{corollary}[main]{Corollary}

\theoremstyle{definition}
\newtheorem{definition}[main]{Definition}

\newtheorem{example}[main]{Example}
\newtheorem{remark}[main]{Remark}

\newcommand{\sse}{\subseteq}

\newcommand{\id}{\mathrm{id}}
\newcommand{\opp}{\mathrm{op}}

\newcommand\C{\mathbb{C}}
\newcommand\R{\mathbb{R}}
\newcommand\N{\mathbb{N}}

\newcommand\Define[1]{\textbf{#1}}
\DeclarePairedDelimiter{\weight}{\lvert}{\rvert}

\DeclareMathSymbol{\mhyphen}{\mathalpha}{operators}{"2D}

\newcommand{\cat}[1]{\mathbf{#1}}
\newcommand{\catC}{\cat{C}}
\newcommand{\catD}{\cat{D}}

\newcommand{\Rpos}{\mathbb{R}_{\ge 0}}

\DeclareMathOperator{\Dom}{Dom}

\newcommand{\Set}{\mathbf{Set}}
\newcommand{\Pfn}{\mathbf{Pfn}}
\newcommand{\EA}{\mathbf{EA}}

\newcommand{\sBBNS}{\boldsymbol{\sigma}\mathbf{BBNS}}
\newcommand{\sBOUS}{\boldsymbol{\sigma}\mathbf{BOUS}}

\newcommand{\OVSu}{\mathbf{OVSu}}

\newcommand{\OVSt}{\mathbf{OVSt}}

\newcommand{\Wstar}{\mathbf{Wstar}}

\newcommand{\pSet}{\mathbf{Set}_{*}}

\newcommand{\sEA}{\boldsymbol{\sigma}\mathbf{EA}}
\newcommand{\omegaBA}{\boldsymbol{\omega}\mathbf{BA}}

\newcommand{\MyNewModLikeCat}[2]{%
  \NewDocumentCommand{#1}{o}{#2\IfValueT{##1}{_{##1}}}}
\MyNewModLikeCat{\WMod}{\mathbf{WMod}}
\MyNewModLikeCat{\CWMod}{\mathbf{CWMod}}
\MyNewModLikeCat{\sWMod}{\boldsymbol{\sigma}\mathbf{WMod}}
\MyNewModLikeCat{\sCWMod}{\boldsymbol{\sigma}\mathbf{CWMod}}
\MyNewModLikeCat{\sEMod}{\boldsymbol{\sigma}\mathbf{EMod}}
\MyNewModLikeCat{\EMod}{\mathbf{EMod}}
\MyNewModLikeCat{\AEMod}{\mathbf{AEMod}}
\MyNewModLikeCat{\Conv}{\mathbf{Conv}}

\newcommand{\Pow}{\mathcal{P}}

\newcommand{\truth}{\mathbf{1}}
\newcommand{\falsity}{\mathbf{0}}
\newcommand{\mpunct}[1]{\,\text{#1}}

\newcommand{\Tot}{\mathrm{Tot}}
\newcommand{\St}{\mathrm{St}}
\newcommand{\sSt}{\St_{\le}}
\newcommand{\Pred}{\mathrm{Pred}}
\newcommand{\Base}{B}
\newcommand{\sBase}{\Base_{\le}}
\newcommand{\pproj}{\mathord{\vartriangleright}}
\newcommand{\copow}{\cdot}

\newcommand{\Tz}{\mathcal{T}} 
\newcommand{\Mlt}{\mathcal{M}}

\newcommand{\smul}{\cdot}

\DeclarePairedDelimiter{\paren}{(}{)}
\DeclarePairedDelimiter{\tup}{\langle}{\rangle}
\DeclarePairedDelimiter{\cotup}{[}{]}
\DeclarePairedDelimiter{\ptup}{\langle\!\langle}{\rangle\!\rangle}

\makeatletter
\newcommand{\my@set@mid}{\mathrel{}\mathclose{}\delimsize\vert\mathopen{}\mathrel{}}
\DeclarePairedDelimiterX{\set}[1]{\lbrace}{\rbrace}{%
  \my@repl@vert@i{\my@set@mid}{#1}}
\NewDocumentCommand{\my@repl@vert@i}{m >{\SplitArgument{1}{|}}m}{\my@repl@vert@i@{#1}#2}
\NewDocumentCommand{\my@repl@vert@i@}{m m m}{\IfNoValueTF{#3}{#2}{#2#1#3}}
\makeatother

\newenvironment{talign}
 {\align}
 {\endalign}
\newenvironment{talign*}
 {\csname align*\endcsname}
 {\endalign}
\makeatletter
\let\talign\align@qed
\expandafter\let\csname talign*@qed\endcsname\align@qed
\makeatother

\newenvironment{Auxproof}{}{}

\makeatletter
\newcommand{\hideauxproof}{%
  \renewenvironment{Auxproof}{\@bsphack\Collect@Body\@gobble}{\@Esphack}%
}
\hideauxproof
\makeatother

\title{Dichotomy between Deterministic and Probabilistic Models\\ in Countably Additive Effectus Theory}
\def\titlerunning{Dichotomy between Deterministic and Probabilistic Models in Countably Additive Effectus Theory}

\author{Kenta Cho
\institute{National Institute of Informatics, Japan}
\email{cho@nii.ac.jp} \and
    Bas Westerbaan
\institute{University College London}
\email{bas@westerbaan.name} \and
    John van de Wetering
\institute{Radboud Universiteit Nijmegen}
\email{wetering@cs.ru.nl}
}
\def\authorrunning{K. Cho, B. Westerbaan \& J. van de Wetering}

\begin{document}

\maketitle

\begin{abstract}
Effectus theory is a relatively new approach to categorical logic
that can be seen as an abstract form of generalized probabilistic
theories (GPTs).  While the scalars of a GPT are always the real
unit interval~$[0,1]$, in an effectus they can form any \emph{effect
monoid}. Hence, there are quite exotic effectuses resulting from
more pathological effect monoids.


In this paper we introduce \emph{$\sigma$-effectuses}, where certain
countable sums of morphisms are defined. We study in particular
$\sigma$-effectuses where unnormalized states can be normalized.
We show that a non-trivial~$\sigma$-effectus with normalization
        has as scalars
            either the two-element effect monoid~$\{0,1\}$
            or the real unit interval~$[0,1]$.
When states and/or predicates separate the morphisms we find that in
the $\{0,1\}$ case the category must embed into the category of sets
and partial functions (and hence the category of Boolean algebras),
showing that it implements a deterministic model,
while in the $[0,1]$ case we find it embeds into the category of 
Banach order-unit spaces and of Banach pre-base-norm
spaces (satisfying additional properties), recovering the structure present in GPTs.

Hence, from abstract categorical and operational considerations we find a dichotomy between deterministic and convex probabilistic models of physical theories.
\end{abstract}

\section{Introduction}

In the widely used generalized probabilistic theories
(GPTs),
see e.g.~\cite{barrett2007information,BarnumBCLS2010,BarnumW2016,CassinelliL2016},
measurement and probability are of central importance. A system in a
GPT is described by a
real vector space
corresponding to the states of the system, while the \emph{effects},
two-outcome measurements, lie in the dual vector space.

Effectus theory, introduced by Jacobs~\cite{jacobs2015new},
is an approach to categorical logic that can describe
deterministic, probabilistic or quantum logic;
see also \cite{cho2015introduction,Cho2019PhD,basthesis}.
An effectus is analogous to a GPT where the real interval $[0,1]$
of probabilities is replaced by an \emph{effect
monoid}~$M$. As a result, states form an \emph{(abstract) convex set} over~$M$ instead of lying in a real vector space, while effects form an \emph{effect module} over~$M$.
Tull~\cite{tull2016,tull2019phdthesis}
showed that effectuses can be understood as
certain \emph{operational theories}
in the style of
Chiribella et al.~\cite{chiribella2010probabilistic,DArianoCP2017}.

Taking the effect monoid of scalars in an effectus to be $[0,1]$, the effectus is quite close in structure to that of a GPT (especially when operationally motivated state/effect separation properties are imposed, cf.~Section~\ref{sec:separation}). 
Instead taking the scalars to be the Booleans $\{0,1\}$, the effectus describes a deterministic theory where every predicate either holds with certainty on each state, or does not hold at all.
Every effect monoid can form the set of scalars of an effectus
(Propositions~\ref{prop:emod-effectus}
and~\ref{prop:wmod-effectus}),
and since there exist quite pathological effect
monoids, there are exotic effectuses that have no easy comparison
to GPTs or deterministic theories.

In this paper we show that this situation changes when we
consider effectuses with some additional structure.
A central notion in effectus theory is the existence of certain sums of morphisms.
In this paper we introduce \emph{$\sigma$-effectuses},
where we strengthen this to the existence of
certain countable sums of morphisms,
based on the well-established notion of
\emph{partially additive categories}~\cite{ArbibM1980,ManesA1986}.
The extension allows
measurements with countably many
outcomes (see Remark~\ref{remark:operational-interpretation}),
and also it generalizes
the assumption that one can form
countable mixture of states
(e.g.\ \cite{Mackey2004,EdwardsG1970,Edwards1970}).
In a~$\sigma$-effectus the scalars form an \emph{$\omega$-complete}
effect monoid (i.e.~where suprema of increasing sequences exist).
In \cite{effectmonoids} these were shown to always embed in a
direct sum of a Boolean algebra and the unit interval of a commutative C$^*$-algebra.
This characterization shows that the scalars in a $\sigma$-effectus
are necessarily well-behaved. This has several immediate
consequences for~$\sigma$-effectuses, such as that the scalars are
always commutative.

A natural condition, which in GPTs is usually assumed implicitly,
is that every unnormalized state can be \emph{normalized}.
We present a number of equivalent conditions for a $\sigma$-effectus to allow normalization, one of which is that the scalars
must be one of~$\{0\}$, $\{0,1\}$ and $[0,1]$.

Hence $\sigma$-effectuses with normalization come in three different
types.  When the scalars of an effectus are $\{0\}$, the category is
equivalent to the trivial single object category, and hence this
type is not particularly interesting.
If instead the scalars are $\{0,1\}$, the $\sigma$-effectus describes a
deterministic theory where each predicate (does not) hold with certainty.
If we additionally assume that states separate morphisms,
every such $\sigma$-effectus $\catC$ has a faithful morphism
of $\sigma$-effectuses into
the category $\Pfn$ of sets and partial functions
(and hence the category of Boolean algebras).
And finally, if the scalars are~$[0,1]$ we have a GPT-like convex
probabilistic theory.
Under suitable separation assumptions,
the $\sigma$-effectus faithfully embeds into
a category of \emph{order-unit spaces} and
of \emph{(pre-)base-norm spaces},
which are ordered vector spaces used in
GPTs~\cite{BarnumW2016,CassinelliL2016}.
Our results then establish, from purely categorical and operational
considerations, a dichotomy between classical deterministic
and convex probabilistic models.

\section{Preliminaries}




We recall
the well-established notions of
partially $\sigma$-additive monoids
and partially $\sigma$-additive categories\footnote{%
Arbib and Manes
called these notions simply `partially additive monoids'
and `partially additive categories'.
We here added `$\sigma$' in order to emphasize their
countable structures and to avoid confusion with
their finitary counterparts.}
due to Arbib and Manes \cite{ArbibM1980,ManesA1986},
and the finitary counterparts of these structures.
Further details can be found in \cite{Cho2019PhD}.

\begin{definition}
\label{def:pcm}
A \Define{partial commutative monoid}
(\Define{PCM})
is a set $X$ with an element $0\in X$ 
and a \emph{partial} binary operation
$\ovee\colon X\times X\rightharpoonup X$
such that for all $x,y,z \in X$
\begin{itemize}
	\item $(x\ovee y)\ovee z = x\ovee(y\ovee z)$ (associativity),
	\item $x\ovee y= y\ovee x$ (commutativity),
	\item and $0\ovee x = x$ (unitality).
\end{itemize}
Here `$=$' is taken to be a \emph{Kleene equality}:
`if either side is defined, then so is the other,
and they are equal'. 
Hence an equation like $x\ovee y = z$ is taken to
mean both that $x\ovee y$ is defined, as well 
as that we have the equality $x\ovee y = z$.
We will write $x\perp y$ to denote $x\ovee y$ is defined.

Let $M,N$ and $L$ be PCMs.
A function $f:M\rightarrow N$ is \Define{additive} if $f(0)=0$ and
$f(x)\ovee f(y) = f(x\ovee y)$ for all $x\perp y$ in $M$.
A function $g\colon M\times N\to L$
is \Define{biadditive}
if $g(x,-)\colon N\to L$
and $g(-,y)\colon M\to L$ are additive
for all $x\in M$ and $y\in N$.
\end{definition}

A finite sequence $x_1,\dotsc,x_n$ in a PCM $M$ is \Define{summable}
if $\bigovee_{i=1}^n x_i \coloneqq
(\dotsb (x_1\ovee x_2)\ovee \dotsb)\ovee x_n$
is defined in $M$.
The sum $\bigovee_{i=1}^n x_i$ does not depend on the ordering,
yielding a partial addition operation on finite families.
Arbib and Manes defined the notion of partial addition
extended to countable families.

\begin{definition}
A \Define{partially $\sigma$-additive monoid} (\Define{$\sigma$-PAM})
is a nonempty set $M$
equipped with
a partial operation $\bigovee$ that sends
a countable family $(x_j)_{j\in J}$ of elements in $M$
to an element $\bigovee_{j\in J}x_j$ in $M$,
satisfying the three axioms below.
We say that $(x_j)_{j\in J}$ is \Define{summable} if
$\bigovee_{j\in J}x_j$ is defined.
\begin{itemize}
\item \Define{Partition-associativity axiom}:
For each countable family $(x_j)_{j\in J}$
and each countable partition
$J=\biguplus_{k\in K} J_k$,
the family $(x_j)_{j\in J}$ is summable if and only if
$(x_j)_{j\in J_k}$ is summable for each $k\in K$
and $(\bigovee_{j\in J_k} x_j)_{k\in K}$ is summable.
In that case, one has
$\bigovee_{j\in J} x_j
=
\bigovee_{k\in K}
\bigovee_{j\in J_k} x_j$.
\item \Define{Unary sum axiom}:
Each singleton $\{x_j\}_{j\in \set{*}}$
is summable and satisfies $\bigovee_{j\in \set{*}} x_j=x_{*}$.
\item \Define{Limit axiom}:
A countable family $(x_j)_{j\in J}$
is summable
whenever
for any finite subset $F\subseteq J$,
the subfamily $(x_j)_{j\in F}$ is summable.
\end{itemize}
Note that every $\sigma$-PAM is a PCM via
$x_1\ovee x_2=\bigovee_{i\in\set{1,2}} x_i$
and $0=\bigovee\emptyset$.

Let $M,N$ and $L$ be $\sigma$-PAMs.
A function $f\colon M\to N$ is \Define{$\sigma$-additive}
if for any summable family $(x_j)_{j\in J}$ in $M$,
the family $(f(x_j))_{j\in J}$ is summable in $N$
and $f(\bigovee_{j\in J}(x_j))=\bigovee_{j\in J} f(x_j)$.
A function $g\colon M\times N\to L$
is \Define{$\sigma$-biadditive}
if $g(x,-)\colon N\to L$
and $g(-,y)\colon M\to L$ are $\sigma$-additive
for all $x\in M$ and $y\in N$.
\end{definition}

Following Arbib and Manes, we will introduce
a notion of categories equipped with
partial addition of morphisms.
But first we require some additional definitions.
To better understand these definitions the reader might consult Examples~\ref{ex:pfn} and~\ref{ex:wstar} that satisfy the assumptions of these definitions.
\begin{definition}
A category $\catC$
is \Define{enriched over PCMs}
(resp.\ \Define{enriched over $\sigma$-PAMs})
if each homset $\catC(A,B)$ is a PCM
(resp.\ $\sigma$-PAM)
and each composition map
$\circ\colon\catC(B,C)\times \catC(A,B)\to \catC(A,C)$
is biadditive (resp.~$\sigma$-biadditive).
\end{definition}

\begin{definition}
A category $\catC$ with zero morphisms $0\colon A\to B$
(such as when it is enriched over PCMs) has for
each coproduct $\coprod_{j\in J} A_j$
 \Define{partial projections}
$\pproj_i\colon \coprod_{j\in J} A_j\to A_i$
characterized by $\pproj_i\circ \kappa_i=\id$
and $\pproj_i\circ \kappa_k=0$ for $k\neq i$.
Here $\kappa_i\colon A_i\to \coprod_{j\in J} A_j$
denote coprojections.
A family $(f_j\colon B\to A_j)_{j\in J}$ of morphisms
in $\catC$ is \Define{compatible} if there exists an
$f\colon B\to \coprod_{j\in J} A_j$ such that
$\pproj_j\circ f= f_j$ for each $j\in J$.
\end{definition}

\begin{definition}
\label{def:sigma-pac}
A \Define{finitely partially additive category}
(resp.\ \Define{partially $\sigma$-additive category})
is a category with finite (resp.\ countable) coproducts 
that is enriched over PCMs
(resp.\ over $\sigma$-PAMs)
satisfying the following two axioms
relating coproducts to the additive structure.
\begin{itemize}
\item \Define{Compatible sum axiom}:
Compatible pairs of morphisms $f,g\colon A\to B$
(resp.\ countable families $(f_j\colon A\to B)_{j\in J}$)
are summable in $\catC(A,B)$.
\item \Define{Untying axiom}:
If $f,g\colon A\to B$ are summable,
then $\kappa_1\circ f,\kappa_2\circ g\colon A\to B+B$
are summable too.
\end{itemize}
We write `finPAC' for `finitely partially additive category' and `$\sigma$-PAC' for `partially $\sigma$-additive category'.
\end{definition}
\begin{remark}
  Fin/$\sigma$-PACs can be characterized in a more categorically simple manner as categories with finite/countable coproducts, zero maps, and some other axioms~\cite[\S\,3.8.1]{Cho2019PhD}, \cite[\S\,5]{ArbibM1980}.
\end{remark}
Before moving on to effectuses, we need a final additional type of structure.

\begin{definition}
\label{def:EA}
    An \Define{effect algebra}~\cite{foulis1994effect} is a PCM $(E,\ovee, 0)$
with a `top' element $1\in E$ such that
for each $a\in E$,
\begin{itemize}
	\item there is a unique $a^\bot\in E$
	(called the orthosupplement) such that $a\ovee a^\bot=1$,
	\item and $a\perp 1$ implies $a=0$.
\end{itemize}
We write $\EA$ for the category of
effect algebras and additive maps.
\end{definition}
Note that effect algebras are posets with the partial order defined by $a\leq b$ iff $a\ovee c = b$ for some $c$.

\begin{example}\label{ex:Boolean-is-effect-algebra}
  A Boolean algebra $(B,0,1,\vee, \perp)$ is an effect algebra with
    $(\ )^\bot$ the regular complement, $a\perp b$ iff $a \wedge b = 0$
  and in that case $a\ovee b \equiv a\vee b$.
\end{example}

\begin{example}
  \label{ex:standard-effect-algebra}
  Let $B(H)$ be the space of bounded operators on a Hilbert space $H$
  equipped with the standard partial order. Its \emph{effects} are
  the operators $A\in B(H)$ satisfying $0\leq A\leq 1$. The space
  of effects $[0,1]_{B(H)}$ is then an effect algebra with $A\perp
  B$ when $A+B \leq 1$ and then $A\ovee B\equiv A+B$.
\end{example}

\begin{remark}\label{partial-remark}
The usual notion of morphisms
$f\colon E\to D$ between
effect algebras requires them to additionally be \Define{unital} in the sense that $f(1)=1$.
Our morphisms in $\EA$ however are only
`subunital', i.e.\ $f(1)\le 1$.
We make this change
because
we will use effectuses \emph{in partial form}
which denotes a category
with `partial' morphisms;
see Remark~\ref{partial-total-remark} below.
We will require a similar change in morphisms
in several other categories.
\end{remark}

\begin{remark}
The definition of an effect algebra might seem a bit arbitrary.
They are however canonical in the following way:
    the category of effect algebras with (unital) morphisms
    is isomorphic to the Eilenberg--Moore category
    for the free-forgetful adjunction between the category of orthomodular posets
        and that of bounded posets~\cite{harding2004remarks,jenvca2015effect}.
\end{remark}

\section{Effectuses and \texorpdfstring{$\sigma$}{sigma}-effectuses}
\label{sec:effectus-sigma-effectus}

In this section, we present the basic theory of $\sigma$-effectuses.
We describe effectuses as well,
showing how their theory~\cite{cho2015introduction,jacobs2015new}
can be naturally extended to the $\sigma$-additive setting.
In addition, we introduce a notion of \emph{($\sigma$-)weight modules}
to axiomatize the structure of substates.


A ($\sigma$-)effectus is basically a fin/$\sigma$-PAC with a special
unit object representing `no system'. The morphisms to the unit object
are then the ways in which a system can be `destroyed' or `measured'
and hence are the effects of the system.
They are assumed to form effect algebras.

\begin{definition}
\label{def:effectus}
    An \Define{effectus}
(in partial form, see Remark~\ref{partial-total-remark} below) is a finPAC $\catC$
with a distinguished `unit' object $I\in\catC$
satisfying
the following conditions.
\begin{myenum}
\item
For each $A\in\catC$,
the hom-PCM $\catC(A,I)$ is an effect algebra.
We write $\truth_A$
and $\falsity_A=0_{A I}$
for the top and bottom in $\catC(A,I)$.
\item
$\truth_B\circ f=\falsity_A$ implies $f=0_{A B}$
for all $f\colon A\to B$.
\item
$\truth_B\circ f\perp \truth_B\circ g$
implies $f\perp g$
for all $f,g\colon A\to B$.
\end{myenum}
A \Define{$\sigma$-effectus} is a $\sigma$-PAC
$\catC$ with a distinguished
object $I\in\catC$ satisfying the same conditions (i)--(iii).

\parpic(3cm,1.9cm)[r]{%
\begin{tikzcd}[row sep=scriptsize,ampersand replacement=\&]
FA\ar[d, "\truth_{FA}"'] \ar[dr, "F\truth_A"] \&
\\
I_{\catD} \ar[r, "u"', "\cong"] \& FI_{\catC}
\end{tikzcd}}
A \Define{morphism of effectuses} (resp.\ \Define{$\sigma$-effectuses})
$(\catC,I_{\catC})\rightarrow (\catD,I_{\catD})$ is a
functor $F\colon \catC\to\catD$
that preserves finite (resp.\ countable) coproducts
and `preserves the unit' in the sense that
there is an isomorphism $u\colon I_{\catD}\to FI_{\catC}$
such that $F\truth_A=u\circ \truth_{FA}$ for each $A\in\catC$. 
I.e.~the diagram on the right commutes.
\end{definition}

There are several types of morphisms in an effectus that have special significance:
\begin{itemize}
	\item A \Define{predicate} on $A$ is any morphism $p\colon A\to I$. 
	We write $\Pred(A)=\catC(A,I)$ for the set of predicates.
	\item A \Define{substate} on $A$ is any morphism $\omega\colon I\to A$.
	We write $\sSt(A) = \catC(I,A)$ for the set of substates.
	\item A morphism $f\colon A\to B$ in a ($\sigma$-)effectus
is \Define{total} if $\truth_B \circ f=\truth_A$. The total morphisms form a (wide) subcategory
$\Tot(\catC)\hookrightarrow\catC$.
\item A \Define{state} on $A$ is a substate that is total. We write $\St(A)=\Tot(\catC)(I,A)$ for the set of states.
\item A \Define{scalar} is a morphism $s\colon I\to I$.
We view these as abstract probabilities.
\end{itemize}

\begin{remark}
\label{remark:operational-interpretation}
As studied by Tull \cite{tull2016,tull2019phdthesis},
one can interpret a ($\sigma$-)effectus as
an \emph{operational theory} in the style
of Chiribella et
al.~\cite{chiribella2010probabilistic,
ChiribellaDP2015,DArianoCP2017}
(see also \cite[\S\,6.1, 6.2]{Cho2019PhD}).
In their terminology,
each morphism $f\colon A\to B$ is called an \Define{event}.
A \Define{test} from system $A$ to $B$
is then a summable family of events
$(f_x\colon A\to B)_{x\in X}$
such that $\bigovee_{x\in X} f_x$ is total.
The indexing set $x\in X$ is understood as
the set of outcomes of the test.
In particular,
a `preparation' test $(\omega_x\colon I\to A)_{x\in X}$
consists of substates and
an `observation' test $(p_x\colon A\to I)_{x\in X}$
consists of predicates.
Each `closed' test $(s_x\colon I\to I)_{x\in X}$,
which satisfies $\bigovee_{x\in X} s_x=1$, describes the
abstract probability $s_x$ that the test yields an outcome $x\in X$.
\end{remark}

\begin{example}\label{ex:pfn}
  A \Define{partial function} $f\colon X\rightharpoonup Y$ is
  a function of sets
  where for each $x\in X$, $f(x)$ is either an element of $Y$ or
  undefined. We write $\Dom(f)\sse X$ for the domain of definition,
  i.e.\ the set of $x\in X$
  where $f(x)$ is defined. Partial functions compose in the obvious
  way. The category of sets and partial functions $\Pfn$ is a
  $\sigma$-effectus with the singleton $I=\set{*}$ as unit.
  Partial functions are summable when they have disjoint
  domains of definition.
  Such partial functions can be merged into one partial function
  in the obvious way, which defines the sum.
  Indeed, $\Pfn$ is the prototypical example of a $\sigma$-PAC
  in \cite{ArbibM1980,ManesA1986}.
  For a set $X$, we have
  $\St(X) \cong X$ and $\Pred(X) \cong \Pow(X)$, the powerset of $X$.
  Finally, the total maps are the partial functions that are defined
  everywhere, and hence $\Tot(\Pfn) \cong \Set$.
\end{example}

\begin{example}\label{ex:wstar}
  Let $\Wstar$ be the category of
  $W^*$-algebras (also known as \emph{von Neumann algebras}) and subunital normal positive linear maps
  (see \cite[\S\,2.6]{Cho2019PhD} for the definitions).
  Then the opposite $\Wstar^{\opp}$
  is a $\sigma$-effectus with $\C$ as unit.
  A family of maps $f_j\colon\mathfrak{A}\rightarrow\mathfrak{B}$
  in $\Wstar^{\opp}$ for $j\in J$ is summable iff
  $\sum_{j\in F}
  f_j(1_{\mathfrak{B}})\le 1_{\mathfrak{A}}$ in $\mathfrak{A}$
  for all finite $F\subseteq J$.
  Then define $(\bigovee_j f_j)(b) = \sum_{j\in J} f_j(b)$
  where the infinite sum converges ultraweakly in $\mathfrak{A}$.
  States on
  $\mathfrak{A}\in\Wstar^{\opp}$ are
  unital normal positive maps from $\mathfrak{A}\to\C$,
  which are
  known as \emph{normal states} in the literature.
  The set of predicates
  $\Pred(\mathfrak{A}) = [0,1]_{\mathfrak{A}}$ is its unit interval.
  The total maps are precisely the unital maps.
  We note that the category of C$^*$-algebras
  similarly forms an effectus, but not a $\sigma$-effectus
  \cite[Example 7.3.36]{Cho2019PhD}.
\end{example}

\begin{remark}\label{partial-total-remark}
  What we defined as an effectus is called an \emph{effectus in
    partial form} in~\cite{cho2015introduction}. It is also possible
  to axiomatize an \emph{effectus in total form}. Given an
  effectus~$C$ in partial form, the subcategory of total
  maps~$\Tot(C)$ is an effectus in total form, which has a final
  object $1=I$. As a total map~$A \to B + 1$ corresponds to a
  (partial) map~$A \to B$, one can define from an effectus in total
  form a category of partial maps, which turns out to recover the
  original effectus in partial form. This correspondence leads to a
  2-categorical equivalence of the relevant categories of
  effectuses~\cite{Cho2015} (see also \cite[\S\,4.2]{Cho2019PhD}).
  We elected to work here with effectuses in partial form because
  the definition admits an obvious extension to the $\sigma$-additive
  case.
  One can define $\sigma$-effectuses in total form
  through the equivalence of the two forms of effectuses,
  but we do not know whether they admit an
  intrinsic categorical characterization
  like effectuses in total form,
  which can be defined in terms of
  pullbacks and jointly monic morphisms
  \cite[Definition 2]{cho2015introduction}.
\end{remark}

By definition,
predicates $p\colon A\to I$ in an effectus form
an effect algebra.
In a $\sigma$-effectus, predicates also have
a $\sigma$-additive structure.
We will show that
the structure of predicates in a $\sigma$-effectus
is captured precisely by the well-established notion of
$\sigma$-effect algebras.

\begin{definition}
\label{def:sigma-effect-algebra}
A \Define{$\sigma$-effect algebra}
\cite{Gudder1998,FoulisG2007}
is an effect algebra
whose partial ordering is \Define{$\omega$-complete},
that is,
where any increasing sequence $a_0\leq a_1\leq \ldots$ has a supremum.
We say a countable family $(x_j)_{j\in J}$ in a $\sigma$-effect algebra $E$
is \Define{summable} when the family $(x_j)_{j\in F}$ is summable
for every finite subset $F\subseteq J$.
For a summable countable family $(x_j)_{j\in J}$
we define $\bigovee_{j\in J} x_j=\bigvee_{F} \bigovee_{j\in F} x_j$
where $F$ runs over all finite subsets of $J$,
and the supremum exists by $\omega$-completeness.
\end{definition}
The definition of sums of countable families 
equips each $\sigma$-effect algebra with 
a canonical $\sigma$-PAM structure
that extends its PCM structure.
Conversely, each effect algebra that is a $\sigma$-PAM is $\omega$-complete.

\begin{proposition}\label{prop:effect-algebra-PAM-iff-complete}
    Let $E$ be an effect algebra with a $\sigma$-PAM structure that extends the PCM structure of $E$. 
    Then $E$ is $\omega$-complete and hence a $\sigma$-effect algebra.
    Moreover, the $\sigma$-PAM structure coincides with the canonical
    $\sigma$-PAM structure of the $\sigma$-effect algebra $E$.
\end{proposition}
\begin{proof}
    See Appendix~\ref{appendix:proofs-sigma-effectus}.
\end{proof}

\begin{corollary}
For any object $A$ in a $\sigma$-effectus $\catC$, $\Pred(A)=\catC(A,I)$
forms a $\sigma$-effect algebra.
\qed
\end{corollary}

The following, straightforwardly verifiable, lemma
establishes the equivalence of two possible notions of morphisms
of $\sigma$-effect algebras.

\begin{lemma}
\label{lem:sigma-additive-equiv-omega-conti}
Let $E,D$ be $\sigma$-effect algebras
and $f\colon E\to D$ an additive map.
Then $f$ is $\sigma$-additive
if and only if
it is \Define{$\omega$-continuous},
i.e.\ if it preserves suprema of increasing sequence
$a_0\le a_1\le \dotsb$.
\qed
\end{lemma}

\subsection{Effect monoids and modules}
The predicates of the unit object $I$ in a ($\sigma$-)effectus do not just form a ($\sigma$-)effect algebra. As they are the morphisms $s:I\rightarrow I$ they also have a `multiplication' operation given by composition of morphisms.
The resulting structure in the finitary case is known as an effect monoid
\cite{Jacobs2011,jacobs2015new}.
We introduce $\sigma$-effect monoids as the counterpart for the countable case.

\begin{definition}\label{def:effectmonoid}
    An \Define{effect monoid} (resp.~\Define{$\sigma$-effect monoid})
    is a ($\sigma$-)effect algebra
    $(M,\ovee, 0, 1)$
    with an associative binary (total) operation
    $\,\cdot\,\colon M\times M\to M$
    that is ($\sigma$-)biadditive
    and satisfies $a\cdot 1 = a = 1\cdot a$
    for all $a\in M$.
Given an effect monoid $M$ we define the \Define{opposite} effect
monoid $M^\opp$ as the same underlying effect algebra, but with the
product defined as $a\cdot'b \equiv b\cdot a$. Obviously $M$ is commutative
iff $M=M^\opp$.
\end{definition}
The monoids
in the symmetric monoidal category of ($\sigma$-)effect algebras with (unital) morphisms
and the
algebraic tensor product
are precisely the ($\sigma$-)effect monoids, hence the name~\cite{jacobs2012coreflections,Gudder1998}. 

The structure of $\omega$-complete effect monoids
has been studied in \cite{effectmonoids}.
It follows
from \cite[Theorem 43]{effectmonoids}
(with Lemma~\ref{lem:sigma-additive-equiv-omega-conti}) that
any $\omega$-complete effect monoid
is a $\sigma$-effect monoid
--- that is, the requirement of $\sigma$-biadditivity
of the multiplication may be weakened to biadditivity.

\begin{example}\label{ex:booleanalgebra}
In $\Pfn$ the scalars are $\{0,1\}$, and hence $\{0,1\}$ is a $\sigma$-effect monoid.
More generally, any Boolean algebra $(B,0,1,\wedge,\vee,(\ )^\perp)$ (being an effect algebra by Example~\ref{ex:Boolean-is-effect-algebra}),
    is an effect monoid with $a\cdot b \equiv a\wedge b$.
Therefore any $\omega$-complete Boolean algebra
is a $\sigma$-effect monoid.
\end{example}

\begin{example}\label{ex:CX}
  The scalars of $\Wstar^{\opp}$ is the real unit interval $[0,1]$,
  which is thus a $\sigma$-effect monoid
  with the usual multiplication and partial addition.
  More generally, let $X$ be a compact Hausdorff space. We denote its space of
    continuous functions into the complex numbers by
    $C(X)\equiv\set{f:X\rightarrow \C| f \text{ continuous}}$. This is
    a commutative unital C$^*$-algebra
    (and conversely by the Gel'fand theorem, any
    commutative C$^*$-algebra with unit is of this form).
    Its unit interval
    $[0,1]_{C(X)} = \set{f:X\rightarrow [0,1]| f \text{ continuous}}$
    is not just an effect algebra but an effect monoid
    (with multiplication defined pointwise).
    The effect monoid~$[0,1]_{C(X)}$ is $\omega$-complete
        (and thus a~$\sigma$-effect monoid)
        if and only if~$X$ is \Define{basically disconnected},
            i.e.~when every cozero set has open
            closure~\cite[1H \& 3N.5]{gillman2013rings}.
\end{example}
These examples of effect monoids are all commutative. In~\cite[Ex.~4.3.9]{Cho2019PhD} and~\cite[Cor.~51]{basmaster}
two different non-commutative effect monoids are constructed.

In the rest of this section,
we study the structures of predicates and substates.
In particular, it will be shown that
any ($\sigma$-)effect monoid can appear as the scalars of
a ($\sigma$-)effectus
(Propositions~\ref{prop:emod-effectus}
and~\ref{prop:wmod-effectus}).

For a monoid $M$,
an \Define{$M$-action} on a set $X$ is
a function ${\cdot}\colon M\times X\to X$
such that $1\cdot x=x$ and $(r\cdot s)\cdot x=r\cdot(s\cdot x)$
for all $r,s\in M$ and $x\in X$.
We will apply this definition to ($\sigma$-)effect monoids.

\begin{definition}
Let $M$ be a ($\sigma$-)effect monoid. 
A \Define{($\sigma$-)effect $M$-module}
is a ($\sigma$-)effect algebra 
$E$
equipped with a ($\sigma$-)biadditive 
$M$-action $\cdot \colon M\times E\to E$.
Explicitly, for example, the biadditivity means:
\[
(r\ovee s)\cdot a= r\cdot a\ovee s\cdot a
\qquad\qquad
r\cdot (a\ovee b)= r\cdot a\ovee r\cdot b
\qquad\qquad  0\cdot a= 0 = r\cdot 0
\]
for all $r,s\in M$ and $a,b\in E$ with $r\perp s$
and $a\perp b$.
We write $\EMod[M]$
(resp.~$\sEMod[M]$)
for the category of ($\sigma$-)effect $M$-modules
and ($\sigma$-)additive 
maps that preserve the $M$-action;
i.e.~$f(r\cdot x)=r\cdot f(x)$.
\end{definition}

\begin{example}
  If $\catC$ is a ($\sigma$-)effectus
  with scalars $M=\catC(I,I)$,
  the set $\Pred(A)$ of predicates
  on $A\in\catC$ is a ($\sigma$-)effect $M$-module,
  with $M$-action given by composition $r\cdot p=r\circ p$.
\end{example}

\begin{example}
    A ($\sigma$-)effect $\{0,1\}$-module
    is just a ($\sigma$-)effect algebra,
    as the $\set{0,1}$-action is trivial.
\end{example}

\begin{example}
    When $M$ is the real unit interval $[0,1]$, an effect $M$-module is precisely a \emph{convex effect algebra}~\cite{gudder1999convex}. These are effect algebras $E$ that are intervals $[0,u]_V$ of ordered vector spaces $V$ with a positive $u\in V$~\cite{gudder1998representation}.
We will come back to this in Section~\ref{sec:OUS}.
\end{example}

\begin{proposition}
\label{prop:emod-effectus}
Let $M$ be an effect monoid (resp.\ $\sigma$-effect monoid).
Then the opposite category
$\EMod[M]^{\opp}$
is an effectus
(resp.\ $\sEMod[M]^{\opp}$ is a $\sigma$-effectus)
with scalars $M$.
The unit object is $M$,
and coproducts are given by Cartesian products
with pointwise operations (which form products in
$\EMod[M]$ and $\sEMod[M]$).
\end{proposition}
\begin{proof}
See \cite[Proposition~3.4.10]{Cho2019PhD}
for the case of effect monoids.
We prove the result for $\sigma$-effect monoids
in Proposition~\ref{prop:sEMod-sigma-effectus}
in Appendix~\ref{appendix:proofs-sigma-effectus}.
\end{proof}

This allows us to describe the assignment
of predicates to each object as a morphism of effectuses.

\begin{proposition}
Let $\catC$ be an effectus
(resp.\ $\sigma$-effectus) with scalars $M=\catC(I,I)$.
Then the assignment $A\mapsto \Pred(A)$
induces a morphism of effectuses $\Pred\colon \catC\to\EMod[M]^{\opp}$
(resp.\ morphism of $\sigma$-effectuses
$\Pred\colon \catC\to\sEMod[M]^{\opp}$).
\end{proposition}
\begin{proof}
See \cite[Lemma~4.2.11]{Cho2019PhD}
for the case of effectuses.
We prove the result for $\sigma$-effectuses
in Proposition \ref{prop:pred-functor-sigma-effectus}
in Appendix \ref{appendix:proofs-sigma-effectus}.
\end{proof}
This mapping from objects to their predicate spaces is the effectus-analogue of the commonly used identification in GPTs of identifying a system with its vector space of effects. Of course, in GPTs we can also identify a system with the vector space of states, this also has an analogue in effectus theory.

The usual approach in effectus theory is to focus on the sets of states,
    which form (abstract) $M$-convex
    sets; see e.g.~\cite{cho2015introduction,jacobs2015new,basthesis}.
However, here we focus on
    the sets of substates
    and axiomatize their structure as
    \emph{($\sigma$-)weight $M$-modules}.
    This is not just natural in the setting of effectuses
    in partial form, but also has the advantage
    that we can avoid technical problems with
    convex sets, see Remark~\ref{remark-convex} below.

\begin{definition}
\label{def:wpmod}
Let $M$ be a ($\sigma$-)effect monoid. 
A \Define{($\sigma$-)weight $M$-module}
is a PCM (resp.\ $\sigma$-PAM) $X$ equipped with
a ($\sigma$-)biadditive 
$M$-action $\cdot\colon M\times X\to X$
and a function $\weight{-}\colon X\to M$,
called the \Define{weight}, such that
\begin{itemize}
\item $\weight{-}\colon X\to M$ is ($\sigma$-)additive 
and preserves the $M$-action,
i.e.\ $\weight{rx}=r\weight{x}$;
\item $\weight{x} = 0$ implies $x=0$;
\item $\weight{x} \perp \weight{y}$ implies $x \perp y$
(resp.\ countable families $(x_j)_{j\in J}$ are summable
when $(\,\weight{x_j}\,)_{j\in J}$ is summable).
\end{itemize}
A function $f\colon X\to Y$
between ($\sigma$-)weight $M$-modules
is \Define{weight-preserving}
if $\weight{f(x)}=\weight{x}$
for all $x\in X$,
and
\Define{weight-decreasing}
if
$\weight{f(x)}\le\weight{x}$
for all $x\in X$.
We denote by
$\WMod[M]$
(resp.\ $\sWMod[M]$)
the category of
($\sigma$-)weight $M$-modules
and weight-decreasing ($\sigma$-)additive
maps that preserves the $M$-action.
\end{definition}

\begin{example}
  If $\catC$ is a ($\sigma$-)effectus
  with scalars $M=\catC(I,I)$,
  the set $\sSt(A)$ of substates
  on $A\in\catC$ is a ($\sigma$-)weight $M^{\opp}$-module,
  with $M^{\opp}$-action given by composition
  (from the right) $r\cdot \omega=\omega\circ r$,
  and weight $\weight{\omega}=\truth\circ \omega$.
  Note that states are precisely
  elements $\omega\in\sSt(A)$ with weight $1$.
\end{example}

For a weight $M$-module $X$,
let $B(X)=\set{x\in X;\ \weight{x}=1 }$
be the set of elements with weight $1$.
The set $B(X)$ is closed under `$M$-convex sums',
i.e.\ $\bigovee^n_{i=1} r_i x_i \in B(X)$
    for~$x_i \in X$ and~$r_i \in M$ with~$\bigovee_i r_i = 1$.
This makes $B(X)$ into an $M$-convex set
\cite[\S\,3.6]{Cho2019PhD}.
In particular, the
states $\St(A)=B(\sSt(A))$ in an effectus
form an $M$-convex set.
In this way, our treatment of
substates subsumes the usual treatment of states
in terms of convex sets.
If $M$ is `well-behaved',
such as when $M=[0,1]$,
the category of $M$-convex sets
is equivalent to
the category of weight $M$-modules
and \emph{weight-preserving} maps
\cite[Proposition 4.4.10]{Cho2019PhD}.

\begin{example}
\label{ex:weight-module-bool}
Both weight $\set{0,1}$-modules
and $\sigma$-weight $\set{0,1}$-modules
are precisely \Define{pointed sets},
i.e.\ sets $X$ equipped with a distinguished element $x_0\in X$.
Every ($\sigma$-)weight $\set{0,1}$-module $X$ is
a pointed set $(X,0)$, and the converse is also true.
This is because in a ($\sigma$-)weight $\set{0,1}$-module,
all nonzero elements have weight $1$ and thus they
cannot be summable with nonzero elements.
This yields isomorphisms of categories
$\WMod[\set{0,1}]\cong \sWMod[\set{0,1}]\cong \pSet$,
where $\pSet$ denotes the category of pointed sets
and functions that preserves the distinguished element.
\end{example}

\begin{proposition}
\label{prop:wmod-effectus}
Let $M$ be an effect monoid (resp.\ $\sigma$-effect monoid).
Then the category
$\WMod[M]$
is an effectus
(resp.\ $\sWMod[M]$ is a $\sigma$-effectus)
with scalars $M$.
The unit object is $M$ and coproducts are given by
$\coprod_{\lambda\in \Lambda} X_{\lambda}
=
\set{(x_{\lambda})_{\lambda}\in \prod_{\lambda\in \Lambda} X_{\lambda}
; (\,\weight{x_{\lambda}}\,)_{\lambda\in\Lambda} \text{ is summable in $M$}}
$ for finite or countable $\Lambda$.
\end{proposition}
\begin{proof}
See \cite[Proposition~3.5.9]{Cho2019PhD}
for the case of effect monoids.
We prove the case of $\sigma$-effect monoids
in Proposition~\ref{prop:sWMod-sigma-effectus}
in Appendix~\ref{appendix:proofs-sigma-effectus}.
\end{proof}

\begin{proposition}
Let $\catC$ be a ($\sigma$-)effectus
with scalars $M=\catC(I,I)$.
The assignment $A\mapsto \sSt(A)$
induces a morphism of effectuses $\sSt\colon \catC\to
\WMod[M^{\opp}]$
(resp.\ morphism of $\sigma$-effectuses $\sSt\colon \catC\to
    \sWMod[M^{\opp}]$.)
\end{proposition}
\begin{proof}
See \cite[Lemma~4.2.11]{Cho2019PhD}
for the case of effectuses.
We prove the result for $\sigma$-effectuses
in Proposition \ref{prop:sst-functor-sigma-effectus}
in Appendix \ref{appendix:proofs-sigma-effectus}.
\end{proof}

\begin{remark}\label{remark-convex}
Similar results to the previous two hold
for $M$-convex sets and states in an effectus
under certain additional assumptions
on the effect monoid $M$ and on the effectus;
see \cite[Corollary 4.4.15 and Proposition 4.5.11]{Cho2019PhD}
and \cite[\S\,3.2.4]{basthesis}.
However, it is an open question
whether the results hold in general.
\end{remark}

\section{Separation properties and normalization}\label{sec:separation}

The definition of a ($\sigma$-)effectus is quite weak. It will
therefore be useful to consider some additional structure that an
effectus might have.
The first structure we consider is based on the notion of `operational
equivalence' used in GPTs
(cf.\ \cite[\S\,2.2]{ChiribellaDP2015}).
This basically says that if two
transformations act the same on all effects or substates that they
must be the same transformations, since they are operationally
indistinguishable.

\begin{definition}
A ($\sigma$-)effectus is \Define{predicate-separated}
when any pair of morphisms $f,g\colon A\to B$
satisfy $f=g$ whenever $p\circ f=p\circ g$
for all $p\in\Pred(B)$.
It is \Define{substate-separated}
when any pair of morphisms $f,g\colon A\to B$
satisfy $f=g$ whenever $f\circ\omega=g\circ \omega$
for all substates $\omega\in\sSt(A)$.
\end{definition}

The following is an immediate consequence
from the definition,
which will be used in
Section \ref{sec:classification}.

\begin{proposition}\label{prop:separation-faithful}
A $\sigma$-effectus $\catC$ is predicate-separated if and only if
the morphism of $\sigma$-effectuses
$\Pred\colon \catC\to \sEMod[M]^{\opp}$
is faithful (as a functor).
It is substate-separated if and only if
the morphism of $\sigma$-effectuses
$\sSt\colon \catC\to \sWMod[M^{\opp}]$
is faithful.\qed
\end{proposition}

Hence, a $\sigma$-effectus satisfying one of
the separation properties can be seen as a `sub-$\sigma$-effectus'
of the $\sigma$-effectus of $\sigma$-effect modules or of
$\sigma$-weight modules.
One could argue that it would be more natural to 
assume state separation, instead of substate separation.
An effectus is \Define{state-separated}
if for any pair of morphisms $f,g\colon A\to B$
we have $f=g$ whenever $f\circ\omega=g\circ \omega$
for all states $\omega\in\St(A)$.
This however turns out to be equivalent to 
substate separation when the next condition 
we introduce is satisfied.

A second property that is usually assumed
(often implicitly) in a GPT is the
possibility of normalizing states
(cf.\ \cite[\S\,4.1.4]{ChiribellaDP2015},
\cite[\S\,5.4.1]{DArianoCP2017}). A `normalized' state $\omega$
is one that has unit probability when the deterministic effect (`always true') is
tested against it: $\truth\circ \omega = 1$. An `unnormalized'
substate can then be interpreted as one that has a probability of
failure at being prepared: $\truth\circ \omega < 1$. Being able to
normalize a state recognizes the possibility of deterministically 
preparing any state that can be probabilistically prepared.

\begin{definition}
A ($\sigma$-)effectus admits \Define{normalization} if
for each nonzero substate $\omega\colon I\to A$,
there exists a unique state
$\bar{\omega}\colon I\to A$ such that
$\omega=\bar{\omega}\circ
(\truth \circ \omega)$.
\end{definition}

\begin{proposition}
\label{prop:state-sep-equiv-substate-sep}
A ($\sigma$-)effectus with normalization
is state-separated if and only if it is substate-separated.
\end{proposition}
\begin{proof}
    See Appendix~\ref{app:proofs-separation-section}.
\end{proof}

In \cite[Proposition~6.4]{Cho2015},
it was shown that if an effectus admits normalization,
the scalars admit a type of division.
In a $\sigma$-effectus, the converse holds,
together with several other equivalent conditions.

\begin{theorem}
\label{thm:normalisation-equiv}
Let $\catC$ be a $\sigma$-effectus.
The following are equivalent.
\begin{myenum}
\item
$\catC$ admits normalization.
\item
The effect monoid $\catC(I,I)$ admits \Define{division}:
for any $s,t\in\catC(I,I)$ with $s\leq t$ and $t\neq 0$,
there is a unique $s/t\in \catC(I,I)$ satisfying $(s/t)\cdot t  = s$.
\item
The effect monoid $\catC(I,I)$ has no nontrivial zero divisors,
i.e.\ $s\cdot t=0$ implies $s=0$ or $t=0$.
\item
Every nonzero scalar $s\colon I\to I$ in $\catC$ is an epi.
\end{myenum}
\end{theorem}
\begin{proof}
    See Appendix~\ref{app:proofs-separation-section}.
\end{proof}

\section{Classification of
\texorpdfstring{$\sigma$}{sigma}-effectuses with normalization}
\label{sec:classification}

In this section, we combine the theory of $\sigma$-effectuses
with the classification result of $\omega$-complete effect monoids
obtained in \cite{effectmonoids}.
It leads to the classification of $\sigma$-effectuses
with normalization:
these $\sigma$-effectuses are either the trivial category,
$\sigma$-effectuses with Boolean scalars $\set{0,1}$,
or $\sigma$-effectuses with probabilistic scalars $[0,1]$.
We then investigate the latter two cases in more detail,
assuming the separation properties.

In Examples~\ref{ex:booleanalgebra} and \ref{ex:CX}
we presented two examples
of $\omega$-complete effect monoids:
$\omega$-complete Boolean algebras
and $[0,1]_{C(X)}$
for basically disconnected compact Hausdorff spaces $X$.
One of the main results of \cite{effectmonoids} shows that
these examples are basically the only possible $\omega$-complete effect monoids.

\begin{theorem}[{\cite[Theorem 54]{effectmonoids}}]\label{thm:effect-monoid-char}
    Let $M$ be an $\omega$-complete effect monoid.
    Then $M$ embeds into $M_1\oplus M_2$, where
    $M_1$ is an $\omega$-complete Boolean algebra,
    and $M_2 =[0,1]_{C(X)}$,
    where $X$ is a basically disconnected compact Hausdorff space. \qed
\end{theorem}

It immediately
follows that any $\omega$-complete effect monoid is commutative,
since both $M_1$ and $M_2$ above are commutative.
Hence we obtain the following result.

\begin{corollary}
    The scalars of a $\sigma$-effectus are commutative.
    \qed
\end{corollary}

Theorem~\ref{thm:effect-monoid-char}
has the following consequence, also shown in \cite{effectmonoids}.

\begin{theorem}[{\cite[Theorem 71]{effectmonoids}}]\label{thm:no-zero-divisors}
    Let $M$ be an $\omega$-complete effect monoid with no non-trivial
    zero divisors.
    Then either $M=\{0\}$, $M=\{0,1\}$ or $M=[0,1]$. \qed
\end{theorem}

Combining Theorems~\ref{thm:no-zero-divisors} and~\ref{thm:normalisation-equiv} we immediately get the following result characterizing the possible scalars in a $\sigma$-effectus with normalization.

\begin{theorem}
A $\sigma$-effectus $\catC$ admits normalization
if and only if the effect monoid $\catC(I,I)$ of scalars is isomorphic
to $\set{0}$, $\set{0,1}$, or $[0,1]$.
\qed
\end{theorem}

Of these three options, the first always leads to a trivial effectus.

\begin{proposition}
  Let $\catC$ be an effectus where the scalars $\catC(I,I)$ are
  isomorphic to $\set{0}$. Then $\catC$ is equivalent to the trivial
  category with a single object and a single morphism.
\end{proposition}
\begin{proof}
Because $\id=0\colon I\to I$, any truth map $\truth\colon A\to I$
satisfies $\truth=\id\circ\truth=0\circ \truth=0$.
Thus for any morphism $f\colon A\to B$
we have $\truth\circ f=0\circ f=0$.
By an axiom of effectuses, we obtain $f=0$.
Therefore for any objects $A,B\in\catC$,
the homset $\catC(A,B)$ is a singleton.
We conclude that
$\catC$ is equivalent to the trivial category.
\end{proof}


\subsection{\texorpdfstring{$\sigma$}{sigma}-Effectus with Boolean scalars}



If a $\sigma$-effectus $\catC$ has Boolean scalars $\{0,1\}$,
the operational theory described by $\catC$ is deterministic:
every predicate either holds with certainty on each state,
or does not hold at all.
Therefore such an effectus is fundamentally classical,
as it is well-known that quantum theory
cannot be described as a deterministic theory.

\begin{example}
Let $\sEA$ be the category of $\sigma$-effect algebras
and $\sigma$-additive maps.
We have $\sEA\cong \sEMod[\set{0,1}]$,
and hence $\sEA^{\opp}$ is an $\sigma$-effectus
with scalars $\set{0,1}$.
Therefore $\sEA^{\opp}$ is deterministic
and `classical'.
It may seem to contradict the fact that
$\sigma$-effect algebras also include spaces of quantum effects.
This paradoxical situation can be explained as follows.

  Let $H$ be a Hilbert space with $\dim(H)>2$,
  and let $E=[0,1]_{B(H)}$ be the set of effects on $H$
  (see Example~\ref{ex:standard-effect-algebra}).
  Then $E$ is a $\sigma$-effect algebra.
  The subset of projections
  $P(H)\sse E$ is then an $\sigma$-effect subalgebra
  and hence is an object in the
  effectus $\sEA^\opp$.
  By the Kochen--Specker theorem~\cite{KochenS1967}, we have
  $\St(P(H))\equiv \Tot(\sEA^\opp)(\set{0,1},P(H)) = \emptyset$,
  that is, there exists
  no unital $\sigma$-additive map $P(H)\to\set{0,1}$.
  This implies $\St(E)=\emptyset$ too.
  Operationally speaking,
  therefore,
  one cannot prepare a system of
  type $P(H)$ or $E$ in $\sEA^\opp$.
  In other words, both $P(H)$ and $E$
  are operationally equivalent to the empty system $0$.
  \begin{Auxproof}
  More technically this means the following.
  Suppose that there are no states on $A$.
  Then $id_A$ and $0_{AA}$ are operationally equivalent.
  From this it follows that
  $A$ is isomorphic to $0$ in the category
  quotiented by the operational equivalence.

  (Note that one can make a similar observation
  for the effectus $\EA^{\opp}$ of effect algebras and additive maps.
  Then $\St(P(H))\equiv
  \Tot(\EA^{\opp})(\set{0,1},P(H)) = \emptyset$
  (for $\dim(H)>2$)
  is precisely the Kochen-Specker theorem,
  see e.g.~\cite{StatonU2018}.)
  \end{Auxproof}
\end{example}

This observation motivates us to restrict ourselves to
$\sigma$-effectuses with scalars $\set{0,1}$
that are substate-separated (or equivalently,
state-separated,
by Proposition~\ref{prop:state-sep-equiv-substate-sep}),
in order to take
operational equivalence into account.
We will show that these $\sigma$-effectuses
always embed into
the $\sigma$-effectus $\Pfn$ of sets and partial functions
via faithful morphisms of $\sigma$-effectuses,
and hence they are `sub-$\sigma$-effectuses' of $\Pfn$.
We also show that they embed into
the $\sigma$-effectus of $\omega$-complete Boolean algebras.
These results make it more precise what we mean by
`$\sigma$-effectuses with scalars $\set{0,1}$ are classical'.

\begin{proposition}\label{prop:equiv-pfn-wmod}
  We have an equivalence of categories
  $\sWMod[\{0,1\}]\overset\simeq\to\Pfn$.
  The functor is also a morphism of $\sigma$-effectuses.
\end{proposition}
\begin{proof}
  As we observed in Example \ref{ex:weight-module-bool},
  $\sigma$-weight $\set{0,1}$-modules are merely pointed sets:
  $\sWMod[\{0,1\}]\cong \pSet$.
  Then the equivalence of the categories
  $\pSet\simeq \Pfn$ is well-known ---
  it sends $f\colon (X,x_0)\to (Y,y_0)$ in $\pSet$
  to $\overline{f}\colon X\setminus\set{x_0}\to Y\setminus\set{y_0}$
  in $\Pfn$ where $\overline{f}(x)$ is defined iff $f(x)\neq y_0$
  and in that case $\overline{f}(x)=f(x)$.
  The equivalence $\sWMod[\{0,1\}]\cong\pSet\simeq \Pfn$
  preserves all coproducts, and it is easily checked that
  it preserves the unit object. Hence it is also
  a morphism of $\sigma$-effectuses.
  \begin{Auxproof}
  Send a $\sigma$-weight $\{0,1\}$-module $X$ to the set
  $X\backslash \{0\}$ in $\Pfn$. A morphism $f:X\rightarrow Y$ in
  $\sWMod[\{0,1\}]$ is sent to
  $\overline{f}:X\backslash \{0\}\rightarrow Y\backslash \{0\}$ where
  $\overline{f}(x) = f(x)$ when $f(x)\neq 0$ and otherwise it is
  undefined.
  In the other direction send a set $A$ in $\Pfn$ to $A+\{0\}$,
  with weight $\weight{a}=1$ for $a\in A$ and $\weight{0}=0$.
  The set $A+\{0\}$ is trivially a $\sigma$-PAM because no family
  that contain more than one element is summable.
  A partial function $f\colon A\rightarrow B$ is sent to
  $\overline{f}:A+\{0\}\rightarrow B+\{0\}$ with $\overline{f}(0) = 0$
  and $\overline{f}(a) = f(a)$ when $f(a)$ is defined and
  $\overline{f}(a)=0$ otherwise for all $a\in A$.
  \end{Auxproof}
\end{proof}

Combining it with Proposition~\ref{prop:separation-faithful},
and with straightforward calculation,
we obtain the following theorem.

\begin{theorem}
Let $\catC$ be a substate-separated $\sigma$-effectus
with $\catC(I,I)\cong\set{0,1}$.
Then there is a faithful morphism of $\sigma$-effectuses
$F\colon\catC\to \Pfn$.
Moreover, we have $\St(A)\cong FA$ for all $A\in\catC$.
\qed
\end{theorem}

We write $\omegaBA$ for the category of
$\omega$-complete\footnote{%
For a Boolean algebra, $\omega$-completeness
is equivalent to
existence of all countable joins (and meets).}
Boolean algebras
and functions that preserves countable joins
and \emph{nonempty} countable meets.
Then one can show that $\omegaBA^{\opp}$
is a $\sigma$-effectus
--- in fact, $\omegaBA$ is a full subcategory of $\sEA$
(the fullness is proved similarly to \cite[Lemma 6.5.18]{Cho2019PhD}).
The following result can be easily verified.



\begin{proposition}\label{prop:powerset-functor}
  The contravariant powerset functor is a faithful morphism of $\sigma$-effectuses $\Pow:\Pfn\rightarrow \omegaBA^\opp$,
where
$\Pow(f)(S)= \set{x\in X | f(x)
\text{ is defined and } f(x)\in S}$
for partial functions $f\colon X\rightharpoonup Y$
and $S\in \Pow(Y)$.\qed
\end{proposition}

Composition of these last two faithful morphisms of $\sigma$-effectuses
yields the following result.

\begin{theorem}
    Let $\catC$ be a substate-separated $\sigma$-effectus with scalars $\{0,1\}$.
    Then there is a faithful morphism of $\sigma$-effectuses $G\colon\catC\rightarrow \omegaBA^\opp$.\qed
\end{theorem}

\noindent
This does \emph{not} mean that
the predicates $\Pred(A)$ form a Boolean algebra,
but rather there is an injection
\[
\Pred(A)
\ \equiv\ 
\catC(A,I)
\ \rightarrowtail \
\omegaBA^\opp(GA,GI)
\ \cong\ 
\omegaBA(\set{0,1},GA)
\ \cong\ 
GA
\mpunct,
\]
so that predicates form a subset of the Boolean algebra $GA$.
In fact, we can prove that
the injection $\Pred(A)\rightarrowtail GA$ is a $\sigma$-additive map.
From this it follows that $\Pred(A)$ is an \emph{orthoalgebra}, i.e.~that
it has the property that $p\perp p$ implies $p=0$.

\begin{Auxproof}
\begin{example}\label{ex:notBoolean}
    Let $E = \{\emptyset, \{1,2\}, \{3,4\}, \{1,3\}, \{2,4\}, \{1,2,3,4\}\}$ be an effect algebra (and hence a $\{0,1\}$-module) with the operations inherited from $P(\{1,2,3,4\})$. Obviously $E$ embeds (as an effect algebra) in the Boolean algebra $P(\{1,2,3,4\})$. Yet $E$ itself is not a Boolean algebra: we see that $\{1,2\} \wedge \{1,3\} = \{1,2\} \wedge \{2,4\} = \emptyset$, and hence $E$'s lattice structure is not distributive: \\$\{1,2\} = \{1,2\} \wedge \{1,2,3,4\} = \{1,2\} \wedge (\{1,3\} \vee \{2,4\}) \neq (\{1,2\} \wedge \{1,3\}) \vee (\{1,2\} \wedge \{2,4\}) = \emptyset$.
\end{example}
\end{Auxproof}

\subsection{\texorpdfstring{$\sigma$}{sigma}-Effectus with probabilistic scalars}
\label{sec:OUS}

In this section we will show that
a $\sigma$-effectus with scalars
$[0,1]$ can be embedded into
the categories of certain ordered vector spaces,
under the assumption of the separation properties.
These ordered vector spaces are
order-unit spaces and (pre-)base-norm spaces,
which serve as abstract spaces of effects and of states,
respectively.
They have long been used in
GPT-style approaches to quantum theory
(also known as `convex operational' approaches);
see e.g.\ \cite{Ludwig1983,Ludwig1985,
DaviesL1970,Edwards1970}
and recent work
\cite{BarnumW2016,CassinelliL2016,Furber2017,Furber2018QPL}.

The embedding results are obtained as consequences of
representation results of
$\sigma$-effect $[0,1]$-modules and (cancellative)
$\sigma$-weight $[0,1]$-modules
into suitable order-unit spaces and (pre-)base-norm spaces.
The proofs of
Propositions~\ref{prop:sBOUS-equiv-sEMod},
\ref{prop:OVSt-equiv-CWMod}, and \ref{prop:sBBNS-equiv-sCWMod}
are deferred to Appendix~\ref{app:convex-proofs}.

We start by recalling the known representation result of
effect $[0,1]$-modules.
\begin{definition}
Let $A$ be an ordered vector space
(with positive cone $A_+$).
An \Define{order unit} of $A$
is a positive element $u\in A_{+}$
such that for all $x\in A$
there exists $n\in\N$ with $-n u \le x \le nu$.
A map $f\colon A\to B$ between
ordered vector spaces with order unit
(say $u_A\in A$ and $u_B\in B$)
is \Define{subunital} if $f(u_A)\le u_B$.
We write $\OVSu$ for the category of
ordered vector spaces with order unit
and subunital positive linear maps.
(A map $f\colon A\to B$ is \Define{positive} if $f(A_+)\subseteq B_+$.)
\end{definition}

Note that for each $(A,u)\in\OVSu$,
the unit interval $[0,u]_A=\set{a\in A| 0\le a \le u}$
is an effect $[0,1]$-module.
Conversely,
for each effect $[0,1]$-module $E$,
one can construct $(A,u)\in\OVSu$
such that $[0,u]_A\cong E$
\cite{gudder1998representation,JacobsMF2016}.
These constructions yield an equivalence of categories.

\begin{proposition}[{\cite[Theorem 14]{JacobsMF2016}}]
\label{prop:OVSu-equiv-EMod}
The functor $\OVSu\to\EMod[[0,1]]$
that sends $(A,u)$ to $[0,u]_A$ is an equivalence of categories.
\qed
\end{proposition}

\begin{definition}
An \Define{order-unit space}
is an ordered vector space
$A$ with order unit $u$
satisfying the Archimedean property:
$nx\le u$ for all $n\in\N$ implies $x\le 0$.
Each order-unit space $(A,u)$
is equipped with the intrinsic
\Define{order-unit norm} given by
$
\norm{a}=\inf\set{r>0 | -ru\le a\le ru}
\mpunct.
$
A \Define{Banach order-unit space} is an
order-unit space that is complete with respect to
the order-unit norm.
\end{definition}

\begin{definition}
An ordered vector space $A$ is
\Define{monotone $\sigma$-complete}
if every ascending sequence $a_0\le a_1\le \dotsb$
in $A$ that is bounded above
has a supremum $\bigvee_{n=0}^{\infty} a_n$.
A map between
monotone $\sigma$-complete ordered vector spaces
is \Define{$\sigma$-normal}
if it preserves suprema of
ascending sequences that are bounded above.
We write $\sBOUS$ for the category of
monotone $\sigma$-complete Banach order-unit spaces
and $\sigma$-normal subunital positive linear maps.
\end{definition}

The equivalence
of Proposition~\ref{prop:OVSu-equiv-EMod}
can be restricted to the following one.

\begin{proposition}
\label{prop:sBOUS-equiv-sEMod}
There is an equivalence of categories $\sBOUS\simeq\sEMod[[0,1]]$.
\end{proposition}

This proves that $\sBOUS^\opp$ is a $\sigma$-effectus.
By Proposition~\ref{prop:separation-faithful},
we obtain the following result.

\begin{theorem}\label{thm:convex-effectus-embedding}
Let $\catC$ be a predicate-separated $\sigma$-effectus with scalars $\catC(I,I)\cong [0,1]$.
Then there is a faithful morphism of $\sigma$-effectuses $F\colon\catC\to \sBOUS^\opp$.
Furthermore, $\Pred(A)\cong [0,u]_{FA}$ for all $A\in \catC$. \qed
\end{theorem}

While this representation onto vector spaces uses the structure of the predicates in the effectus, we can dually find a representation using the structure of the states.
For this we will need a representation of ($\sigma$-)weight $[0,1]$-modules.

\begin{definition}
An \Define{ordered vector space with trace}\footnote{%
It is called a \emph{base ordered linear space}
in \cite{Pumpluen2002}
and a \emph{semi-base-norm space} in \cite{Cho2019PhD}.}
is an ordered vector space $V$
that is positively generated (i.e.\ $V=V_+-V_+$)
and equipped with
a linear functional $\tau\colon V\to \R$ 
called the \Define{trace}
that is \Define{strictly positive}
in the sense that $x>0$ implies $\tau(x)>0$.
A map $f\colon V\to W$ between
ordered vector spaces with trace
is \Define{trace-decreasing} if
$\tau_W(f(x))\le \tau_V(x)$ for all $x\in V_+$.
We write $\OVSt$ for the category of
ordered vector spaces with trace
and trace-decreasing positive linear maps.
\end{definition}

Each $(V,\tau)\in \OVSt$
defines a weight $[0,1]$-module
via its \Define{subbase}
$\sBase(V) =\set{x\in V_+ | \tau(x)\le 1}$, 
with weight $\weight{x}=\tau(x)$.
Clearly, $\sBase(V)$ is \Define{cancellative}
in the sense that $x\ovee y=x\ovee z$ implies $y=z$.
Writing $\CWMod[[0,1]]\hookrightarrow\WMod[[0,1]]$
for the full subcategory of cancellative weight $[0,1]$-modules,
we obtain a functor $\sBase\colon \OVSt\to\CWMod[[0,1]]$.
Conversely, for any cancellative weight $[0,1]$-module $X$
we can construct $V\in\OVSt$ such that $\sBase(V)\cong X$,
giving rise to an equivalence of categories.

\begin{proposition}
\label{prop:OVSt-equiv-CWMod}
The functor $\sBase\colon \OVSt\to\CWMod[[0,1]]$
is an equivalence of categories.
\end{proposition}

Each $(V,\tau)\in\OVSt$
is equipped with an intrinsic seminorm given by:
\[
\norm{x}\ =\ 
\inf\set{\tau(x_1)+\tau(x_2) |
x_1,x_2\in V_{+}
\;\;\textnormal{such that}\;\;
x=x_1-x_2}
\mpunct.
\]
Following Furber \cite{Furber2017},
we call $(V,\tau)$ a \Define{pre-base-norm space}
if the seminorm $\norm{-}$ is a norm
(i.e.\ $\norm{x}=0$ implies $x=0$).
It is a \Define{Banach pre-base-norm space}
if $V$ is complete with respect to the base norm.
To formulate the results below,
we introduce additional (non-standard) terminology.
A Banach pre-base-norm space
has a \Define{$\sigma$-closed subbase}
if for each countable family $(x_n)_{n\in\N}$
in $\sBase(V)$ with $\sum_{n\in\N}\tau(x_n)\le 1$,
the series $\sum_{n=0}^{\infty} x_n$
converges to an element in $\sBase(V)$.\footnote{%
This property is equivalent to
the assumption of the theorem of
Edwards and Gerzon~\cite{EdwardsG1970}.}

We write $\sBBNS\hookrightarrow \OVSt$
for the full subcategory of
Banach pre-base-norm spaces
with a $\sigma$-closed subbase,
and $\sCWMod[[0,1]]\hookrightarrow\sWMod[[0,1]]$
for the full subcategory of cancellative
$\sigma$-weight $[0,1]$-modules.
The equivalence of Proposition~\ref{prop:OVSt-equiv-CWMod}
can be restricted to these categories.

\begin{proposition}
\label{prop:sBBNS-equiv-sCWMod}
There is an equivalence of categories $\sBBNS\simeq\sCWMod[[0,1]]$.
\end{proposition}

As $\sCWMod[[0,1]]$ is a full subcategory of $\sWMod[[0,1]]$, it
is a $\sigma$-effectus, and hence so is $\sBBNS$.
Combining Propositions~\ref{prop:sBBNS-equiv-sCWMod} and~\ref{prop:separation-faithful} we have the following result.

\begin{theorem}\label{thm:convex-effectus-embedding-2}
Let $\catC$ be a state-separated $\sigma$-effectus with scalars $[0,1]$
such that substates $\sSt(A)$ are cancellative.
Then there is a faithful morphism of $\sigma$-effectuses $G\colon\catC\to \sBBNS$.
Furthermore, $\sSt(A)\cong \sBase(GA)$ for all $A\in \catC$.
\qed
\end{theorem}

\begin{remark}
Cancellativity of the substates follows when the effectus is predicate-separated, and hence any state- and predicate-separated $\sigma$-effectus with scalars $[0,1]$ embeds into both $\sBBNS$ and $\sBOUS^\opp$.
\end{remark} 

\section{Conclusion}

We introduced the notion of a $\sigma$-effectus and showed that when they allow normalization of states,
the scalars must be equal to $\{0\}$, $\{0,1\}$, or $[0,1]$. 
The first case was shown to lead to a trivial effectus. 
In the latter two cases we found that when 
operationally motivated state- and/or predicate-separation properties 
are satisfied, in the $\{0,1\}$ case the effectus embeds into 
the category of sets and partial functions, and thus is classical and 
deterministic, while in the $[0,1]$ case $\sigma$-effectuses embed 
into either a category of Banach order-unit spaces, 
or of Banach pre-base-norm spaces. 
We hence have found a dichotomy between deterministic and probabilistic models of physical theories from abstract categorical considerations.

For future work it might be interesting to consider what can be said about $\sigma$-effectuses when the normalization condition is dropped, which would allow for more complex scalars that can also represent `spatial' systems as in \cite{moliner2017space}.

A further open problem that needs to be addressed is whether the nice
categorical definition of an effectus in total form can be modified to
give a notion of an `$\sigma$-effectus in total form'
(see Remark~\ref{partial-total-remark}).
If this is the
case, then our results imply a natural categorical characterization of
Banach order-unit and pre-base-norm spaces.

\paragraph{Acknowledgements}

KC is supported by
ERATO HASUO Metamathematics for Systems Design Project
(No.~JPMJER1603), JST\@.

\bibliographystyle{eptcs}
\bibliography{main}

\appendix



\section{Proofs in Section \ref{sec:effectus-sigma-effectus}}
\label{appendix:proofs-sigma-effectus}

\begin{proof}[Proof of Proposition~\ref{prop:effect-algebra-PAM-iff-complete}]
\newcommand{\tbigovee}{\mathop{\tilde{\bigovee}}}
We write $\tbigovee$
for the given $\sigma$-PAM operation on $E$.
Let $a_0\leq a_1\leq \dotsb$ be an increasing sequence in $E$.
Let $b_0=a_0$ and
$b_{n+1}= a_{n+1}\ominus a_n$ for each $n\in\N$.
Then we have $a_{n}=\bigovee_{k\le n} b_k$,
and in particular, every finite subfamily
of $(b_{n})_{n\in\N}$ is summable.
Therefore the sum
$\tbigovee_{n\in\N} b_{n}$ exists.
We will prove that $\tbigovee_{n\in\N} b_{n}$
is a supremum of $(a_n)_n$.
We have
\[
\tbigovee_{n\in\N} b_{n}
\ =\ \qty\Big(\bigovee_{k\le n} b_n)\ovee \Big(\tbigovee_{k>n} b_k\Big)
\ =\ a_n\ovee \qty\Big(\tbigovee_{k>n} b_k)
\mpunct,
\]
so that $\tbigovee_{n\in\N} b_{n}$
is an upper bound of $(a_n)_n$.
Suppose that $c$ is an upper bound of $(a_n)_n$.
Then $\bigovee_{k\le n} b_k\le c$
for any $n\in\N$, and hence
the sequence $c^{\bot},b_0,\dotsc,b_n$
is summable for any $n\in \N$.
This implies that the sum
$c^{\bot}\ovee(\tbigovee_{n\in \N} b_n)$ exists.
Hence $\tbigovee_{n\in \N} b_n\le c$,
as desired.
Therefore $E$ is $\omega$-complete.
To verify that $\tbigovee$ coincides with the
canonical $\sigma$-PAM structure,
let $(x_j)_{j\in J}$ be a summable countable family.
If $J$ is finite,
it is clear that
$\tbigovee_j x_j=\bigovee_j x_j$.
If $J$ is infinite,
then we may assume $J=\N$
without loss of generality.
Then the same argument as above
proves
$\tbigovee_{n\in\N} x_{n}=\bigvee_{n\in\N}\bigovee_{k\le n} x_k$,
and the right-hand side coincides with
canonical $\bigovee_{n\in\N} x_n$.
\end{proof}

\begin{Auxproof}
\begin{lemma}
\label{lem:pam-equiv}
Let $M$ be a PCM.
Let $\bigovee$ be a
partial function that sends
countable families in $M$ to elements in $M$
such that the restriction of $\bigovee$
to finite families coincides with
the finite partial sum in the PCM $M$.
Then $(M,\bigovee)$ is a $\sigma$-PAM
if and only if
it satisfies the limit axiom
and the following `weak partition-associativity axiom':
\begin{itemize}
\item
For each countable family $(x_j)_{j\in J}$
and each countable partition $J=\biguplus_{k\in K} J_k$,
if $(x_j)_{j\in J}$ is summable, then
the family $(x_j)_{j\in J_k}$ is summable for each $k\in K$;
$(\bigovee_{j\in J_k} x_j)_{k\in K}$ is summable;
and
\[
\bigovee_{j\in J} x_j
=
\bigovee_{k\in K}
\bigovee_{j\in J_k} x_j
\mpunct.
\]
\end{itemize}
\end{lemma}
\begin{proof}
The unary sum axiom holds by
the assumption that $\bigovee$ is compatible
with the PCM structure.
We only need to show that
for each countable family
$(x_j)_{j\in J}$
and each countable partition $J=\biguplus_{k\in K} J_k$,
if the family $(x_j)_{j\in J_k}$ is summable for each $k\in K$
and $(\bigovee_{j\in J_k} x_j)_{k\in K}$ is summable,
then $(x_j)_{j\in J}$ is summable.
By the limit axiom,
it suffices to show that
$(x_j)_{j\in F}$
is summable for each finite subset $F\subseteq J$.
By the weak partition-associativity,
the sums $y_k\coloneqq \bigovee_{j\in J_k\cap F}
x_j$
and $z_k\coloneqq \bigovee_{j\in J_k\setminus F}
x_j$ are defined, and
$y_k\ovee z_k=\bigovee_{j\in J_k}x_j$.
Let $G=\set{k\in K| J_k\cap F\neq \emptyset}$.
Then also
\[
\bigovee_{k\in G} \bigovee_{j\in J_k}x_j
=\bigovee_{k\in G} y_k\ovee z_k
\]
is defined.
Since $G$ is finite,
the associativity of the PCM
implies that the sums
\[
\bigovee_{k\in G} y_k=
\bigovee_{k\in G}
\bigovee_{j\in J_k\cap F} x_j
=\bigovee_{j\in F} x_j
\]
are all defined, and hence
$(x_j)_{j\in F}$ is summable.
\end{proof}

\begin{lemma}
\label{lem:charact-eff-sigma}
An effectus $\catC$
is a $\sigma$-PAC, hence a $\sigma$-effectus,
if and only if the following conditions hold.
\begin{myenum}
\item\label{lem:charact-eff-sigma.coprod}
$\catC$ has countable coproducts.
\item\label{lem:charact-eff-sigma.jm}
For each object $A$
and each countable set $J$,
the partial projections
$\pproj_j\colon J \cdot A\to A$
from the copower $J \copow A$ of $A$ by $J$
(i.e.\ the $J$-fold coproduct) are jointly monic.
\item\label{lem:charact-eff-sigma.compat}
Let $(f_j\colon A\to B)_{j\in J}$
be a countable family of parallel morphisms.
If the family $(f_j\colon A\to B)_{j\in F}$
is compatible for each finite subset $F\subseteq J$,
then $(f_j\colon A\to B)_{j\in J}$ is compatible.
\end{myenum}
\end{lemma}
\begin{proof}
If an effectus $\catC$ satisfies
\ref{lem:charact-eff-sigma.coprod} and \ref{lem:charact-eff-sigma.jm},
then we can define a partial sum operation
on countable families $(f_j\colon A\to B)_{j\in J}$
by:
sum $\bigovee_{j\in J} f_j$
is defined when $(f_j)_j$ is compatible;
and in that case
$\bigovee_{j\in J} f_j=\nabla\circ f$,
where $f\colon A\to J\copow B$
is a unique morphism that witnesses compatibility
of $(f_j)_j$.
Note that composition
preserves the partial sums defined in this way,
in each argument.
Thus to show that $\catC$ is a $\sigma$-PAC,
we only need to check that each homset $\catC(A,B)$
is a $\sigma$-PAM.
Note that \ref{lem:charact-eff-sigma.compat}
says that the limit axiom holds.
Since the countable partial sum $\bigovee$
extends the PCM structure of $\catC$,
it suffices to show that it the
`weak partition-associativity axiom'
in Lemma~\ref{lem:pam-equiv} holds.
Let $(f_j\colon A\to B)_{j\in J}$ be a summable family,
compatible via $f\colon A\to J\cdot B$.
Let $\bigoplus_{k\in K} J_k=J$ be a partition.
Then for each $k\in K$,
$(f_j\colon A\to B)_{j\in J_k}$ is compatible
via
\[
A\xrightarrow{f} J\copow B
\xrightarrow{\pproj_{J_k}}
J_k\copow B
\]
where $\pproj_{J_k}$ is the obvious projection map.
Moreover, the family
$(\bigovee_{j\in J_k} f_j\colon A\to B)_{k\in K}$
is compatible via
\[
A\xrightarrow{f} J\copow B
\cong \coprod_{k\in K} J_k\copow B
\xrightarrow{\coprod_k \nabla}
\coprod_{k\in K} B
= K\copow B
\mpunct.
\]
It is then straightforward to verify
$\bigovee_{j\in J} f_j
=
\bigovee_{k\in K}
\bigovee_{j\in J_k} f_j$.
\end{proof}
\end{Auxproof}

To prove that $\sEMod[M]^{\opp}$
and $\sWMod[M]$ are $\sigma$-effectuses,
we use the following characterization of $\sigma$-effectuses
(cf.\ a characterization of $\sigma$-PACs given in \cite[\S\,5]{ArbibM1980}).

\begin{lemma}
\label{lem:charact-sigma-effectus}
Let $\catC$ be a category with
a distinguished object $I$ and
a family of maps $\truth_A\colon A\to I$.
Then $(\catC,I)$
forms an $\sigma$-effectus with truth maps $\truth_A\colon A\to I$
if and only if
the following hold.
\begin{myenum}
\item
$\catC$ has countable coproducts.
\item
$\catC$ has zero morphisms.
\item\label{lem:charact-sigma-effectus.jm}
For each object $A$
and each countable set $J$,
the partial projections
$\pproj_j\colon J \copow A\to A$
from the copower of $A$ by $J$
(i.e.\ the $J$-fold coproduct) are jointly monic.
\item\label{lem:charact-sigma-effectus.compat}
Let $(f_j\colon A\to B)_{j\in J}$
be a countable family of parallel morphisms.
If the family $(f_j\colon A\to B)_{j\in F}$
is compatible for each finite subset $F\subseteq J$,
then $(f_j\colon A\to B)_{j\in J}$ is compatible.
\item
$\truth_{A+B}=\cotup{\truth_A,\truth_B}\colon A+B\to I$
for all $A,B$.
\item
$\truth_B\circ f=0_{A I}$ implies $f=0_{A B}$
for all $f\colon A\to B$.
\item
For all $f,g\colon A\to B$,
if $\truth_B\circ f, \truth_B\circ g\colon A\to I$
are compatible,
then $f, g$ are compatible too.
\item
For each $p\colon A\to I$,
there exists a unique $p^{\bot}\colon A\to I$
such that $p,p^{\bot}$ are compatible and
$\nabla_I\circ\ptup{p,p^{\bot}}=\truth_A$,
where $\nabla_I\colon I+I\to I$ is the codiagonal
and $\ptup{p,p^{\bot}}\colon A\to I+I$ is a unique
(by \textup{\ref{lem:charact-sigma-effectus.jm}})
map satisfying $\pproj_1\circ \ptup{p,p^{\bot}}=p$
and $\pproj_2\circ \ptup{p,p^{\bot}}=p^{\bot}$.
\end{myenum}
\end{lemma}
\begin{proof}
It is straightforward to verify the `only if' direction.
Conversely, when $\catC$ satisfies (i)--(viii), we define
addition on morphisms as follows.
A countable family of morphisms
$(f_j\colon A\to B)$ is summable iff it is compatible.
In that case, by the joint monicity
condition~\ref{lem:charact-sigma-effectus.jm},
there is a unique morphism $f\colon A\to J\cdot B$
such that $f_j=\pproj_j\circ f$ for all $j\in J$.
Then we define the sum by $\bigovee_{j\in J} f_j\coloneqq\nabla\circ f$,
where $\nabla\colon J\cdot B\to B$ is the codiagonal.
It is not hard to verify
that this addition on each homset satisfies the axioms of $\sigma$-PAMs,
$\sigma$-PACs, and $\sigma$-effectuses.
The details can be found in
\cite[Proposition~3.8.6 and Lemma~7.3.38]{Cho2019PhD}.
\end{proof}

\begin{proposition}
\label{prop:sEMod-sigma-effectus}
Let $M$ be a $\sigma$-effect monoid.
Then the opposite category $\sEMod[M]^{\opp}$ is a $\sigma$-effectus.
\end{proposition}
\begin{proof}
We invoke Lemma~\ref{lem:charact-sigma-effectus}.
We take $I=M$ and $\truth_E\colon E\to M$
in $\sEMod[M]^{\opp}$ to be
the map $\truth_E\colon M\to E$
in $\sEMod[M]$ given by $\truth_E(s)=s\cdot 1$
\begin{myenum}
\item
$\sEMod[M]$ has all products
given by Cartesian products $\prod_j E_j$
with operations defined pointwise.
Thus $\sEMod[M]^{\opp}$ has all coproducts.
\begin{Auxproof}
Let $(E_{\lambda})_{\lambda\in\Lambda}$
be a family of $\sigma$-effect $M$-modules.
We claim that the Cartesian product
$\prod_{\lambda}E_{\lambda}$ with pointwise operations
is a product in $\sEMod[M]$.
Note that the $M$-action on $\prod_{\lambda} E_{\lambda}$
indeed preserves countable sums:
\begin{talign*}
(\bigovee_j r_j)\smul (a_{\lambda})_{\lambda}
&=((\bigovee_j r_j)\smul a_{\lambda})_{\lambda}
&&\text{by def.\ of $\smul$}
\\
&=(\bigovee_j r_j\smul a_{\lambda})_{\lambda}
\\
&=\bigovee_j (r_j\smul a_{\lambda})_{\lambda}
&&\text{by def.\ of $\bigovee$}
\\
&=\bigovee_j r_j\smul(a_{\lambda})_{\lambda}
&&\text{by def.\ of $\smul$}
\\
r\smul \bigovee_j (a_{\lambda j})_{\lambda}
&=r\smul (\bigovee_j a_{\lambda j})_{\lambda}
&&\text{by def.\ of $\bigovee$}
\\
&=(r\smul \bigovee_j a_{\lambda j})_{\lambda}
&&\text{by def.\ of $\smul$}
\\
&=(\bigovee_j r\smul a_{\lambda j})_{\lambda}
\\
&=\bigovee_j (r\smul a_{\lambda j})_{\lambda}
&&\text{by def.\ of $\bigovee$}
\\
&=\bigovee_j r\smul (a_{\lambda j})_{\lambda}
&&\text{by def.\ of $\smul$}
\end{talign*}
Let $(f_{\lambda}\colon B\to
E_{\lambda})_{\lambda\in \Lambda}$
be a family of morphisms in $\sEMod[M]$.
Then we define $\tup{f_{\lambda}}_{\lambda}\colon
D\to\prod_{\lambda}E_{\lambda}$
by $\tup{f_{\lambda}}_{\lambda}(b)=
(f_{\lambda}(b))_{\lambda}$.
Then $\tup{f_{\lambda}}_{\lambda}$
preserves countable sums:
\[\textstyle
\tup{f_{\lambda}}_{\lambda}
(\bigovee_j b_j)
=
(f_{\lambda}(\bigovee_j b_j))_{\lambda}
=
(\bigovee_j f_{\lambda}(b_j))_{\lambda}
=
\bigovee_j (f_{\lambda}(b_j))_{\lambda}
\mpunct.
\]
\end{Auxproof}
\item
The constant zero functions
are zero morphisms in $\sEMod[M]$,
and hence in $\sEMod[M]^{\opp}$.
\begin{Auxproof}
$0_{ED}\colon E\to D$ is a $\sigma$-additive $M$-module map:
\begin{talign*}
0_{ED} (\bigovee_j a_j)
&= 0_{D}
= \bigovee_j 0_D
= \bigovee_j 0_{ED} ( a_j)
\\
0_{ED} (ra)
&= 0_{D}
= r\smul 0_{D}
= r\smul 0_{ED} (a)
\end{talign*}
\end{Auxproof}
\item
Let $J$ be a countable set.
The partial projections
$\pproj_j\colon J\copow E\to E$
in $\sEMod[M]^{\opp}$
are
morphisms
$\pproj_j\colon E\to E^J$
in $\sEMod[M]$ that send
$x\in E$ to the $J$-tuple that has $0$
at every coordinate except $x$
at the $j$th coordinate.
If $f,g\colon E^J\to D$ in $\sEMod[M]$ satisfy
$f\circ \pproj_j=g\circ \pproj_j$
for all $j\in J$,
then
\[
f((x_j)_j)
=f(\bigovee\nolimits_{j} \pproj_j(x_j))
=\bigovee\nolimits_{j} f (\pproj_j(x_j))
=\bigovee\nolimits_{j} g (\pproj_j(x_j))
=\dotsb
=g((x_j)_j)
\mpunct.
\]
Therefore the maps $\pproj_j$
are jointly epic in $\sEMod[M]$
and hence jointly monic in the opposite.
\item
We prove that
a countable family
$(f_j\colon E\to D)_{j\in J}$
in $\sEMod[M]^{\opp}$
is compatible if and only if
$(f_j(1))_{j\in J}$ is summable in $E$.
By the limit axiom in $E$, this implies
\ref{lem:charact-sigma-effectus.compat}
of Lemma~\ref{lem:charact-sigma-effectus}.
Let $(f_j)_{j\in J}$ be a compatible family.
Then in $\sEMod[M]$,
there exists
a map $f\colon D^J\to E$ such that $f\circ\pproj_j=f_j$.
Since $(1)_{j\in J}\in D^J$ can be written
as $\bigovee_{j\in J} \pproj_j(1)$,
it follows that
the sum $\bigovee_{j\in J} f_j(1)=f(\bigovee_{j\in J} \pproj_j(1))$
is defined.
Conversely, if $(f_j(1))_{j\in J}$ is summable,
define a map $\ptup{f_j}_j\colon D^J\to E$
by $\ptup{f_j}_j((a_j)_j)=\bigovee_j f_j(a_j)$.
We can prove that $\ptup{f_j}_j\colon D^J\to E$
is a morphism in $\sEMod[M]$.
\begin{Auxproof}
\begin{align*}
\ptup{f_j}_j(\bigovee\nolimits_{k} (a_{jk})_j)
&=\ptup{f_j}_j((\bigovee\nolimits_{k} a_{jk})_j)
\\
&=\bigovee\nolimits_j f(\bigovee\nolimits_{k} a_{jk})
\\
&=\bigovee\nolimits_j \bigovee\nolimits_{k}  f(a_{jk})
\\
&=\bigovee\nolimits_{k} \bigovee\nolimits_j  f(a_{jk})
\\
&= \bigovee\nolimits_{k} \ptup{f_j}_j((a_{jk})_j)
\mpunct.
\end{align*}
It preserves the scalar multiplication:
\begin{align*}
\ptup{f_j}_j(r\smul (a_j)_j)
&=\ptup{f_j}_j((r\smul a_j)_j)
\\
&=\bigovee\nolimits_j f_j(r\smul a_j)
\\
&=\bigovee\nolimits_j r\smul f_j(a_j)
\\
&= r\smul \bigovee\nolimits_j f_j(a_j)
\\
&= r\smul\ptup{f_j}_j((a_j)_j)
\mpunct.
\end{align*}
\end{Auxproof}
Then $(f_j)_{j\in J}$ is compatible via $\ptup{f_j}_j$.
\item
$\tup{\truth_E,\truth_D}(s)
=(\truth_E(s),\truth_D(s))
=(s\cdot 1,s\cdot 1)
=s \cdot (1,1)
=\truth_{E\times D}(1) $
in $\sEMod[M]$
and thus $\cotup{\truth_E,\truth_D}=\truth_{E+ D}$
in $\sEMod[M]^{\opp}$.
\item
Let $f\colon E\to D$ be a morphism with
$\truth_D \circ f=0$ in $\sEMod[M]^{\opp}$.
It is a morphism $f\colon D\to E$ in $\sEMod[M]$,
which satisfies $0=0(1)=(f\circ \truth_D)(1)=f(1)$.
Then for any $a\in D$ we have $0\le f(a)\le f(1)=0$
and therefore $f$ is the constant zero function.
\item Let $f,g\colon E\to D$
be morphisms in $\sEMod[M]^{\opp}$
such that $\truth_D\circ f$ and $\truth_D\circ g$ are compatible.
By the characterization of the compatibility in point (iv),
$f(1)=(f\circ \truth_D)(1)$ and $g(1)=(g\circ \truth_D)(1)$
are summable in $E$.
Again by this characterization, $f$ and $g$ are compatible.
\item
This holds because
$p\mapsto p(1)$ defines
a bijection $\sEMod[M](M,E)\cong E$
that sends $\truth_E\colon M\to E$ to $1\in E$
and preserves sums $\ovee$,
where the sums in $\Pred(E)$
are defined by $p\ovee q\coloneqq\nabla_I\circ \ptup{p,q}$.
\qedhere
\end{myenum}
\end{proof}

\begin{proposition}
\label{prop:pred-functor-sigma-effectus}
Let $\catC$ be an effectus with scalars $M=\catC(I,I)$.
Then the assignment $A\mapsto \Pred(A)$
induces a morphism of $\sigma$-effectuses $\Pred\colon \catC\to\sEMod[M]^{\opp}$.
\end{proposition}
\begin{proof}
The well-definedness of the functor $\Pred\colon \catC\to\sEMod[M]^{\opp}$
is easy. It preserves the unit object:
we have
$\Pred(I)= \catC(I,I)=M$
and $\Pred(\truth_A)=(-)\circ \truth_A=\truth_{\Pred(A)}$.
It sends countable coproducts in $\catC$ to products in $\sEMod[M]$:
\[
\Pred(\coprod\nolimits_{\lambda} E_{\lambda})
\ =\ \catC(\coprod\nolimits_{\lambda} E_{\lambda},I)
\cong
\prod\nolimits_{\lambda}\catC(E_{\lambda},I)
\ =\ 
\prod\nolimits_{\lambda}\Pred(E_{\lambda})
\mpunct.
\]
It is easy to see that the bijection is indeed an isomorphism in $\sEMod[M]$.
\end{proof}

\begin{proposition}
\label{prop:sWMod-sigma-effectus}
Let $M$ be an $\sigma$-effect monoid.
Then the category $\sWMod[M]$ is a $\sigma$-effectus.
\end{proposition}
\begin{proof}
We invoke Lemma~\ref{lem:charact-sigma-effectus}.
We take $I=M$ and
define $\truth_X\colon X\to M$
in $\sWMod[M]$ by $\truth_X(x)=\weight{x}$.
\begin{myenum}
\item
First we show that
$\sWMod$ has countable coproducts.
For a countable family $(X_{\lambda})_{\lambda\in \Lambda}$ of objects,
we define the underlying set by
\[
\coprod_{\lambda\in \Lambda} X_{\lambda}
\ =\ 
\set[\Big]{(x_{\lambda})_{\lambda}\in \prod_{\lambda\in \Lambda} X_{\lambda}
| (\,\weight{x_{\lambda}}\,)_{\lambda\in\Lambda} \text{ is summable in $M$}}
\]
and the weight of $(x_{\lambda})_{\lambda}\in \coprod_{\lambda\in \Lambda} X_{\lambda}$ by
$\weight{(x_{\lambda})_{\lambda}}=
\bigovee_{\lambda\in \Lambda}\weight{x_{\lambda}}$.
This determines summability in $\coprod_{\lambda\in \Lambda} X_{\lambda}$:
a countable family $((x_{\lambda j})_{\lambda})_{j\in J}$
is summable if $(\,\weight{(x_{\lambda j})_{\lambda}}\,)_{j\in J}
=(\bigovee_{\lambda} \weight{x_{\lambda j}})_{j\in J}$
is summable in $M$.
We define the $\sigma$-PAM structure and
and $M$-action pointwise.
It is straightforward to verify
that $\coprod_{\lambda\in \Lambda} X_{\lambda}$
is a $\sigma$-weight module,
and that it is a coproduct with
coprojections $\kappa_{\lambda}\colon X_{\lambda}\to\coprod_{\lambda\in \Lambda} X_{\lambda}$
that sends each element $x\in X_{\lambda}$ to
the $\Lambda$-tuple with $0$ everywhere except $x$ at the $\lambda$th
coordinate.

\begin{Auxproof}
\textit{(The operations are well-defined.)}
Let $((x_{\lambda j})_{\lambda})_j$ be a summable family
of elements in $\coprod_{\lambda\in \Lambda} X_{\lambda}$.
This means that
$(\bigovee_{\lambda} x_{\lambda j})_{j\in J}$ is summable
in $M$.
Then for each $\lambda\in \Lambda$, the family $(x_{\lambda j})_j$
is summable in $X_{\lambda}$, and hence one can define
the sum pointwise: $\bigovee_j (x_{\lambda j})_{\lambda}\coloneqq
(\bigovee_j x_{\lambda j})_{\lambda}$.
The sum is a member of $\coprod_{\lambda\in \Lambda} X_{\lambda}$
because $\bigovee_{\lambda}\weight{\bigovee_j x_{\lambda j}}=
\bigovee_{\lambda}\bigovee_j \weight{x_{\lambda j}}$
is defined in $M$.

For any $(x_{\lambda})_{\lambda}\in \coprod_{\lambda\in \Lambda} X_{\lambda}$
and $r\in M$,
the tuple
$(rx_{\lambda})_{\lambda}$ is a member of $\coprod_{\lambda\in \Lambda} X_{\lambda}$
because $\bigovee_{\lambda}\weight{rx_{\lambda}}=\bigovee_{\lambda} r\weight{x_{\lambda}}=
r\cdot \bigovee_{\lambda} \weight{x_{\lambda}}$ is defined in $M$.

\textit{(The operations satisfy the axioms.)}
For a summable family $((x_{\lambda j})_{\lambda})_j$
in $\coprod_{\lambda\in \Lambda} X_{\lambda}$,
\begin{talign*}
\weight{\bigovee_j (x_{\lambda j})_{\lambda}}
&= \weight{(\bigovee_j x_{\lambda j})_{\lambda}}\\
&= \bigovee_{\lambda} \weight{\bigovee_j x_{\lambda j}}\\
&= \bigovee_{\lambda} \bigovee_j \weight{x_{\lambda j}}\\
&= \bigovee_j \bigovee_{\lambda} \weight{x_{\lambda j}}\\
&=\bigovee_j \weight{(x_{\lambda j})_{\lambda}}
\end{talign*}
and for $(x_{\lambda})_{\lambda}\in\coprod_{\lambda\in \Lambda} X_{\lambda}$
and $r\in M$,
\begin{talign*}
\weight{r\cdot (x_{\lambda})_{\lambda}}
&=\weight{(r\cdot x_{\lambda})_{\lambda}}\\
&=\bigovee_{\lambda}\weight{r\cdot x_{\lambda}}\\
&=\bigovee_{\lambda} r\cdot \weight{x_{\lambda}}\\
&=r\cdot \bigovee_{\lambda} \weight{x_{\lambda}}\\
&=r\cdot \weight{(x_{\lambda})_{\lambda}}
\end{talign*}
Thus weight $\weight{-}$
is a $\sigma$-additive $M$-module map,
which moreover reflect summability by definition.
It reflects zero, too:
if $\weight{(x_{\lambda})_{\lambda}}=\bigovee_{\lambda}\weight{x_{\lambda}}=0$,
then for each $\lambda\in \Lambda$, one has $\weight{x_{\lambda}}=0$
and hence $x_{\lambda}=0$, so $(x_{\lambda})_{\lambda}=0$.
Clearly the Unary Sum Axiom holds.
The Limit Axiom holds since
summability is defined via weight.
To see the Partition-Associativity,
let $((x_{\lambda j})_{\lambda})_{j\in J}$ be a countable family
in $\coprod_{\lambda\in \Lambda} X_{\lambda}$
Let $\biguplus_{k\in K} J_k=J$ be a countable partition.
Then
\begin{talign*}
&((x_{\lambda j})_{\lambda})_{j\in J} \quad\text{summable}
\\
&\iff
(\,\weight{(x_{\lambda j})_{\lambda}}\,)_{j\in J}
\quad\text{summable in $M$}
\\
&\iff
\forall k\in K\ldotp
(\,\weight{(x_{\lambda j})_{\lambda}})_{j\in J_k}
\quad\text{summable in $M$}
\quad\text{\&}\quad
(\,\bigovee_{j\in J_k} \weight{(x_{\lambda j})_{\lambda}}\,)_{k\in K}
\quad\text{summable in $M$}
\\
&\iff
\forall k\in K\ldotp
((x_{\lambda j})_{\lambda})_{j\in J_k}
\quad\text{summable}
\quad\text{\&}\quad
(\,\weight{\bigovee_{j\in J_k} (x_{\lambda j})_{\lambda}}\,)_{k\in K}
\quad\text{summable in $M$}
\\
&\iff
\forall k\in K\ldotp
((x_{\lambda j})_{\lambda})_{j\in J_k}
\quad\text{summable}
\quad\text{\&}\quad
(\bigovee_{j\in J_k} (x_{\lambda j})_{\lambda})_{k\in K}
\quad\text{summable}
\end{talign*}
In that case,
\[
\bigovee_{j\in J}(x_{\lambda j})_{\lambda}
=
\bigovee_{k\in K}
\bigovee_{j\in J_k}(x_{\lambda j})_{\lambda}
\]
since the sum is defined pointwise.

Moreover we have
\begin{talign*}
(\bigovee_{j} r_j)\cdot (x_{\lambda})_{\lambda}
&= ((\bigovee_{j} r_j)\cdot x_{\lambda})_{\lambda}
\\
&= (\bigovee_{j} r_j\cdot x_{\lambda})_{\lambda}
\\
&= \bigovee_{j} (r_j\cdot x_{\lambda})_{\lambda}
\\
&=\bigovee_{j} r_j\cdot (x_{\lambda})_{\lambda}
\\
r\cdot \bigovee_j (x_{\lambda j})_{\lambda}
&=r\cdot (\bigovee_j x_{\lambda j})_{\lambda}
\\
&=(r\cdot \bigovee_j x_{\lambda j})_{\lambda}
\\
&=(\bigovee_j r\cdot x_{\lambda j})_{\lambda}
\\
&=\bigovee_j (r\cdot x_{\lambda j})_{\lambda}
\\
&=\bigovee_{j} r\cdot (x_{\lambda j})_{\lambda}
\end{talign*}
Thus $M$-action is $\sigma$-biadditive.

Finally, we show that $\coprod_{\lambda}X_{\lambda}$ is a coproduct.
It is straightforward to check that
the coprojections
$\kappa_j\colon X_{\lambda}\to \coprod_{\lambda}X_{\lambda}$
are morphisms in $\sWMod[M]$.
In fact, they are weight-preserving.
To show the universality of the coproduct,
let $f_{\lambda}\colon X_{\lambda}\to Y$
be morphisms in $\sWMod$.
We define a cotuple
$\cotup{f_{\lambda}}_{\lambda}\colon
\coprod_{\lambda}X_{\lambda}\to X_{\lambda}$ by
\[
\cotup{f_{\lambda}}_{\lambda}((x_{\lambda})_{\lambda})=
\bigovee\nolimits_{\lambda} f_{\lambda}(x_{\lambda})
\mpunct.
\]
The sum is defined because
$(\,\weight{x_{\lambda}}\,)_{\lambda}$ is summable in $M$
and $\weight{f_{\lambda}(x_{\lambda})}\le
\weight{x_{\lambda}}$ and hence
$(\,\weight{f_\lambda(x_{\lambda})}\,)_{\lambda}$ is summable in $M$.
The cotuple $\cotup{f_{\lambda}}_{\lambda}$
is weight-decreasing:
\[
\weight{\cotup{f_{\lambda}}_{\lambda}((x_{\lambda})_{\lambda})}
\ =\ 
\weight{\bigovee\nolimits_{\lambda} f_{\lambda}(x_{\lambda})}
\ =\ 
\bigovee\nolimits_{\lambda} \weight{f_{\lambda}(x_{\lambda})}
\le
\bigovee\nolimits_{\lambda} \weight{x_{\lambda}}
\ =\ 
\weight{(x_{\lambda})_{\lambda}}
\]
Thus it preserves summability.
It preserves sums:
\begin{talign*}
\cotup{f_{\lambda}}_{\lambda}(
\bigovee_j (x_{\lambda j})_{\lambda})
&=
\cotup{f_{\lambda}}_{\lambda}(
(\bigovee_j x_{\lambda j})_{\lambda})
\\
&=
\bigovee_{\lambda} f_{\lambda}
(\bigovee_j x_{\lambda j})
\\
&=
\bigovee_{\lambda}\bigovee_j
f_{\lambda}(x_{\lambda j})
\\
&=
\bigovee_j
\bigovee_{\lambda}
f_{\lambda}(x_{\lambda j})
\\
&=
\bigovee_j
\cotup{f_{\lambda}}_{\lambda}(
(x_{\lambda j})_{\lambda})
\end{talign*}
It preserves $M$-action:
\begin{talign*}
\cotup{f_{\lambda}}_{\lambda}(
r\cdot  (x_{\lambda})_{\lambda})
&=
\cotup{f_{\lambda}}_{\lambda}(
(r\cdot  x_{\lambda})_{\lambda})
\\
&=
\bigovee_{\lambda}
f_{\lambda}(r\cdot  x_{\lambda})
\\
&=
\bigovee_{\lambda}
r\cdot f_{\lambda} (x_{\lambda})
\\
&=
r\cdot (\bigovee_{\lambda}
f_{\lambda} (x_{\lambda}))
\\
&=
r\cdot \cotup{f_{\lambda}}_{\lambda}(
 (x_{\lambda})_{\lambda})
\end{talign*}
Therefore $\cotup{f_{\lambda}}_{\lambda}$
is a morphism in $\sWMod$.
It is clear that
$\cotup{f_{\lambda}}_{\lambda}\circ\kappa_{\lambda}=
f_{\lambda}$
Suppose that a morphism
$g\colon \coprod_{\lambda} X_{\lambda}\to Y$
in $\sWMod$
satisfies
$g\circ\kappa_{\lambda}= f_{\lambda}$
for all $\lambda$.
Then
\begin{talign*}
g((x_{\lambda})_{\lambda})
&=g(\bigovee_{\lambda} \kappa_{\lambda}(x_{\lambda}))\\
&=\bigovee_{\lambda}
g(\kappa_{\lambda}(x_{\lambda}))\\
&=\bigovee_{\lambda} f_{\lambda}(x_{\lambda})\\
&=\cotup{f_{\lambda}}_{\lambda}((x_{\lambda})_{\lambda})
\end{talign*}
\end{Auxproof}

\item
The constant zero functions $0\colon X\to Y$
form zero morphisms in $\sWMod[M]$.

\item
The partial projections
$\pproj_{k}\colon
J\copow X\to X$
for $k\in J$ are given
by $\pproj_{k}((x_{j})_{j})=x_{k}$.
It is clear that these maps are jointly monic.

\item
Let $(f_{j}\colon X\to Y)_{j\in J}$ be a countable
family of morphisms in $\sWMod[M]$.
We claim that $(f_{j})_{j}$
is compatible if and only if
$\bigovee_{j}\weight{f_j(x)}$ is defined and
$\bigovee_{j}\weight{f_j(x)}\le \weight{x}$
for all $x\in X$.
This implies
\ref{lem:charact-sigma-effectus.compat}
of Lemma~\ref{lem:charact-sigma-effectus},
because $\bigovee_{j}\weight{f_j(x)}$
is the supremum of the sums $\bigovee_{j\in F}\weight{f_j(x)}$
for finite subsets $F\subseteq J$.
Suppose that $(f_{j})_{j}$
is compatible via $f\colon X\to J\copow Y$.
Then
for each $x\in X$,
one has
$\pproj_j(f(x))=
f_j(x)$,
and thus by definition of $\pproj_j$,
we have $f(x)=(f_{j}(x))_{j}$.
As $f$ is weight-decreasing,
\[
\weight{x}\ \ge\ \weight{f(x)}\ =\ \weight{(f_{j}(x))_{j}}
\ =\ \bigovee\nolimits_{j}\weight{f_x(x)}
\mpunct.
\]
Conversely, if $\bigovee_{j}\weight{f_j(x)}\le \weight{x}$
for all $x\in X$,
then we can show that the map $f\colon X\to J\cdot Y$
given by $f(x)=(f_{j}(x))_{j}$
is a well-defined morphism in $\sWMod[M]$
and that $(f_{j})_{j\in J}$ is compatible via $f$.
\begin{Auxproof}
In particular we have
\begin{talign*}
f(\bigovee_k x_k)
&= (f_j(\bigovee_k x_k))_j
\\
&= (\bigovee_k f_j(x_k))_j
\\
&= \bigovee_k (f_j(x_k))_j
\\
&= \bigovee_k f(x_k)
\end{talign*}
\end{Auxproof}

\item
$\truth_{X+Y}(x,y)=
\weight{(x,y)}=
\weight{x}\ovee\weight{y}=
\truth_X(x)\ovee\truth_Y(y)=
\cotup{\truth_{X},\truth_{Y}}(x,y)$.
\item
Suppose that $f\colon X\to Y$
satisfies $\truth_Y\circ f=0$.
For each $x\in X$,
we then have $\weight{f(x)}=0$ and hence $f(x)=0$.
Therefore $f=0$.
\item
Let $f,g\colon X\to Y$
be morphisms
such that $\truth_Y\circ f$ and $\truth_Y\circ g$ are compatible.
By the characterization of the compatibility
in point (iv),
$\weight{(\truth_Y\circ f)(x)}\ovee \weight{(\truth_Y\circ g)(x)}\le \weight{x}$
for all $x\in X$.
Hence $\weight{f(x)}\ovee \weight{g(x)}\le \weight{x}$ for all $x\in X$.
By the same characterization again, we have $f\perp g$.
\item
Let $p\in\sWMod[M](X,M)$.
Define $p^{\bot}\colon X\to M$
by $p^{\bot}(x)=\weight{x}\ominus p(x)$,
where $\weight{x}\ominus p(x)$ is the unique element in $M$
satisfying $(\weight{x}\ominus p(x))\ovee p(x)=\weight{x}$.
It is straightforward to check that
$p^{\bot}$ is a morphism in $\sWMod[M]$,
and a unique one that satisfies the required condition.
\begin{Auxproof}
Then it is $\sigma$-additive:
\begin{talign*}
p^{\bot}(\bigovee_j x_j)
&=
\weight{\bigovee_j x_j}
\ominus
p(\bigovee_j x_j)
\\
&=
\bigovee_j \weight{x_j}
\ominus
\bigovee_j p(x_j)
\\
&\overset{\star}=
\bigovee_j (\weight{x_j} \ominus p(x_j))
\\
&=
\bigovee_j p^{\bot}(x_j)
\end{talign*}
Here equality $\overset{\star}=$ holds
since
\[\textstyle
(\bigovee_j (\weight{x_j} \ominus p(x_j)))\ovee\bigovee_j p(x_j)
=
(\bigovee_j (\weight{x_j} \ominus p(x_j)))\ovee p(x_j))
=
\bigovee_j \weight{x_j}
\mpunct.
\]
$p^{\bot}$ preserves the $M$-action:
\[
p^{\bot}(rx)
=\weight{rx}\ominus p(rx)
=\weight{rx}\ominus p(rx)
\]
\end{Auxproof}
\qedhere
\end{myenum}
\end{proof}

The following lemma is the countable version
of \cite[Lemma~4.8]{Cho2015}
(or \cite[Lemma~3.2.5]{Cho2019PhD}).
It can be proved in the same manner as the finite case.

\begin{lemma}
\label{lem:decomposition-bijection}
Let $\catC$ be a $\sigma$-effectus,
and $\coprod_{\lambda\in\Lambda} B_{\lambda}$
a countable coproduct in $\catC$.
There is a bijective correspondence
between morphisms $f\colon A\to \coprod_{\lambda\in\Lambda} B_{\lambda}$
and families of morphisms
$(f_{\lambda}\colon A\to B_{\lambda})_{\lambda\in\Lambda}$
such that $(\truth \circ f_{\lambda})_{\lambda\in\Lambda}$
is summable in $\Pred(A)=\catC(A,I)$.
They are related via $f_{\lambda}=\pproj_{\lambda}\circ f$.
\qed
\end{lemma}
\begin{Auxproof}
First note that the family
$(\kappa_{\alpha}\circ \pproj_{\alpha}\colon
\coprod_{\lambda} B_{\lambda}\to
\coprod_{\lambda} B_{\lambda})_{\alpha\in\Lambda}$
is compatible via
$\coprod_{\alpha}\kappa_{\alpha}\colon
\coprod_{\lambda} B_{\lambda}
=
\coprod_{\alpha} B_{\alpha}
\to
\coprod_{\alpha}
(\coprod_{\lambda} B_{\lambda})$,
and hence the family is summable.
But then
$\bigovee_{\lambda}\kappa_{\lambda}\circ \pproj_{\lambda}
=\id_{\coprod_{\lambda} B_{\lambda}}$
is verified by composing $\kappa_{\lambda}\colon
B_{\lambda}\to \coprod_{\lambda} B_{\lambda}$
to each side.
Now, given $f\colon A\to \coprod_{\lambda} B_{\lambda}$,
define $f_{\lambda}\coloneqq \pproj_{\lambda}\circ f$.
Then
\[\textstyle
\truth\circ f
= \truth\circ (\bigovee_{\lambda}\kappa_{\lambda}\circ \pproj_{\lambda})
\circ f
= \bigovee_{\lambda}\truth\circ \kappa_{\lambda}\circ \pproj_{\lambda}\circ f
= \bigovee_{\lambda}\truth\circ f_{\lambda}
\mpunct,
\]
which shows that
$(\truth \circ f_{\lambda})_{\lambda}$ is summable.
because $\truth\circ\kappa_{\lambda}\circ f_{\lambda}=
\truth\circ f_{\lambda}$.

Conversely,
let $(f_{\lambda}\colon A\to B_{\lambda})_{\lambda\in\Lambda}$
be a family such that
$(\truth \circ f_{\lambda})_{\lambda\in\Lambda}$
is summable in $\catC(A,I)$.
Because $\truth\circ\kappa_{\lambda}\circ f_{\lambda}=
\truth\circ f_{\lambda}$,
the family
$(\truth \circ \kappa_{\lambda}\circ f_{\lambda})_{\lambda}$
is summable too.
By the axioms of a $\sigma$-effectus,
it follows that the family
$(\kappa_{\lambda}\circ f_{\lambda}\colon A\to \coprod_{\lambda} B_{\lambda})_{\lambda}$ is summable.
Thus we can define $f=\bigovee \kappa_{\lambda}\circ f_{\lambda}$.
It is easy to verify that the two constructions
are bijective.
\end{Auxproof}

\begin{proposition}
\label{prop:sst-functor-sigma-effectus}
Let $\catC$ be an $\sigma$-effectus with scalars $M=\catC(I,I)$.
Then the assignment $A\mapsto \sSt(A)$
induces a morphism of
$\sigma$-effectuses $\sSt\colon \catC\to\sWMod[M^{\opp}]$.
\end{proposition}
\begin{proof}
It is easy to see that the functor $\sSt\colon \catC\to
\sWMod[M^{\opp}]$ is well-defined.
It preserves the unit object as
 $\sSt(I)=\catC(I,I)=M$ and the truth maps as
$\sSt(\truth_X)=\truth_X\circ (-)=\weight{-}=\truth_{\sSt(X)}$.
Lastly, it also preserves countable coproducts:
we have a bijection between the underlying sets
\begin{align*}
\sSt(\coprod\nolimits_{\lambda} X_{\lambda})
\coloneqq\catC(I,\coprod\nolimits_{\lambda} X_{\lambda})
&
\;\;\overset{\mathclap{\text{Lem.~\ref{lem:decomposition-bijection}}}}{\cong}
\;\;
\set{(\omega_{\lambda})_{\lambda} \in
\prod\nolimits_{\lambda}\catC(I,X_{\lambda}) |
(\truth\circ\omega_{\lambda})_{\lambda} \text{ is summable in $\catC(I,I)$} }
\\
&\;\;=\;\;
\set{(\omega_{\lambda})_{\lambda} \in
\prod\nolimits_{\lambda}\sSt(X_{\lambda}) |
(\,\weight{\omega_{\lambda}}\,)_{\lambda} \text{ is summable in $M$} }
\\
&\;\;\eqqcolon\;\; \coprod\nolimits_{\lambda} \sSt(X_{\lambda})
\end{align*}
The bijection is indeed
an isomorphism in $\sWMod[M^{\opp}]$.
\begin{Auxproof}
$f\colon
\sSt(\coprod_{\lambda} X_{\lambda})\to \coprod_{\lambda} \sSt(X_{\lambda})$
is given by $f(\omega)=(\pproj_{\lambda}\circ\omega)_{\lambda}$.
It is weight-preserving:
\begin{talign*}
\weight{f(\omega)}
&=\bigovee_{\lambda}\weight{\pproj_{\lambda}\circ\omega} \\
&=\bigovee_{\lambda}\truth\circ\pproj_{\lambda}\circ\omega \\
&=\bigovee_{\lambda}\truth\circ\kappa_{\lambda}
\circ\pproj_{\lambda}\circ\omega \\
&=\truth\circ(\bigovee_{\lambda}\kappa_{\lambda}
\circ\pproj_{\lambda})\circ\omega \\
&=\truth\circ\id \circ\omega \\
&=\truth\circ\omega =\weight{\omega}
\end{talign*}
$\sigma$-additivity and preservation
of $M$-action are easy.
\end{Auxproof}
\end{proof}

\section{Proofs in Section~\ref{sec:separation}}
\label{app:proofs-separation-section}

\begin{proof}[Proof of Proposition~\ref{prop:state-sep-equiv-substate-sep}]
The `only if' direction is obvious.
For the `if' direction,
suppose that the effectus is substate-separated.
Let $f,g\colon A\to B$ be morphisms
such that $f\circ \omega=g\circ \omega$
for any $\omega\in\St(A)$.
We need to show that then $f=g$.
By substate separation it suffices to show that $f\circ \rho = g\circ \rho$ for all \emph{substates} $\rho\in\sSt(A)$.
Hence, let $\rho\in\sSt(A)$ be an arbitrary substate.
If $\rho=0$, then $f\circ\rho=0=g\circ \rho$.
Otherwise, if $\rho\neq 0$, let $\overline{\rho}$ be the normalization
of $\rho$, i.e.\ the state satisfying
$\overline{\rho}\circ \truth\circ \rho=\rho$.
By assumption on $f$ and $g$ we have $f\circ \overline{\rho} = g\circ\overline{\rho}$ and hence
$
f\circ \rho = f\circ \overline{\rho}\circ \truth\circ \rho
= g\circ \overline{\rho}\circ \truth\circ \rho
= g\circ \rho
$
as desired.
\end{proof}

\begin{proof}[Proof of Theorem~\ref{thm:normalisation-equiv}]
(i) $\Longrightarrow$ (ii):
Already holds for regular effectuses; see~\cite[Proposition~6.4]{Cho2015}
or \cite[Proposition~4.5.2]{Cho2019PhD}.

(ii) $\Longrightarrow$ (iii):
Suppose that $s\cdot t=0$ and $t\neq 0$.
As $s\cdot t \leq t$ there is a unique $(s\cdot t)/t$ 
satisfying $((s\cdot t)/t) \cdot t = s\cdot t = 0$.
But as both $0$ and $s$ have this property we conclude
that $s=(s\cdot t)/t = 0$.

(iii) $\Longrightarrow$ (i):
Let $\omega\colon I\to A$
be a nonzero substate.
We write $s\coloneqq (\truth\omega)^{\bot}$
and define
\[
\tilde{\omega}
\coloneqq\bigovee_{n=0}^{\infty}
\omega\circ s^n
\quad
\colon I\longrightarrow A
\mpunct.
\]
The sum is the \emph{iteration}
of the map $\kappa_1\circ s \ovee
\kappa_2\circ \omega\colon I\to I+A$
and hence exists, see \cite[Theorem~3.2.24]{ManesA1986}.

We prove that $\tilde\omega$ is the normalization of $\omega$.
First, we show that
$\tilde{\omega}$ is a state, i.e.\ a total map.
Let
$t\coloneqq\truth\circ \tilde{\omega}
\ =\  \bigovee_{n=0}^{\infty} s^{\bot}\cdot s^n$.
Then
\[
t\ =\ \bigovee_{n=0}^{\infty} s^{\bot}\cdot s^n
\ =\ s^{\bot}\ovee \qty\Big(\bigovee_{n=0}^{\infty} s^{\bot}\cdot s^n)\cdot s
\ =\ s^{\bot}\ovee t\cdot s
\mpunct.
\]
Since
$t=t\cdot (s\ovee s^{\bot})=t\cdot s\ovee t\cdot s^{\bot}$,
we obtain
$t\cdot s^{\bot}=s^{\bot}$
by cancellation.
Then $s^{\bot}=(t\ovee t^{\bot})\cdot s^{\bot}
=s^{\bot}\ovee (t^{\bot}\cdot s^{\bot})$,
so that $t^{\bot}\cdot s^{\bot}=0$.
Because $s^{\bot}=\truth\omega\neq 0$
and there are no nontrivial zero divisors,
$t^{\bot}=0$, that is,
$\truth\circ \tilde{\omega}=t=1$.
Next, we have
\[
\tilde{\omega}\cdot \truth\omega
\ =\ 
\bigovee_{n=0}^{\infty}\omega\cdot s^n\cdot s^{\bot}
\ =\ 
\omega\cdot
\bigovee_{n=0}^{\infty}s^{\bot}\cdot s^n
\ =\ 
\omega\cdot
1
\ =\ \omega
\mpunct.
\]
Here note that $s$ and $s^{\bot}$ commute.
To see the uniqueness of the normalization,
let $\rho$ be a state with $\omega=\rho\cdot \truth\omega
\;(=\rho\cdot s^{\bot})$.
Then
\[
\rho
\ =\ \rho\cdot 1
\ =\ \rho\cdot \paren[\Big]{\bigovee_{n=0}^{\infty} s^{\bot}\cdot s^n}
\ =\ \bigovee_{n=0}^{\infty} \rho\cdot s^{\bot}\cdot s^n
\ =\ \bigovee_{n=0}^{\infty} \omega\cdot s^n
\ =\ \tilde{\omega}
\mpunct.
\]
Therefore $\tilde{\omega}$ is the normalization
of $\omega$.

(iv) $\Longrightarrow$ (iii):
Let $s\cdot t=0$ for $s,t\in\catC(I,I)$.
Assume $t\neq 0$.
Because $t$ is an epi and
$s\circ t=0=0\circ t$,
we obtain $s=0$. This proves (iii).

(i) $\Longrightarrow$ (iv):
By what we have already proved,
we may assume that (ii) and (iii) hold.
Let $s\colon I\to I$ be a nonzero scalar.
Suppose that $\omega_1\circ s=\omega_2\circ s$
for $\omega_1,\omega_2\colon I\to A$.
If $\omega_1=0$, then
$\truth\circ\omega_2\circ s=0$.
Since $s$ is nonzero, we obtain $\truth\circ\omega_2=0$
by (iii), and hence $\omega_2=0$.
Similarly $\omega_2=0$ implies $\omega_1=0$.
Therefore it suffices to consider the case where
both $\omega_1$ and $\omega_2$ are nonzero.
Let
\[
t\ \coloneqq\  \truth\circ \omega_1\circ s\ =\ \truth\circ\omega_2\circ s
\mpunct.
\]
By (iii) it follows that $t$ is nonzero.
By division,
we have $\truth \omega_1=t/s = \truth \omega_2$.
By normalization,
there are states $\bar{\omega_1},\bar{\omega_2}\colon I\to X$
such that
$\omega_1=\bar{\omega_1}\circ\truth \omega_1$
and $\omega_2=\bar{\omega_2}\circ\truth \omega_2$.
Then
\begin{align*}
\bar{\omega_1}\circ t
&\ =\  \bar{\omega_1}\circ \truth \omega_1\circ s \\
&\ =\  \omega_1\circ s \\
&\ =\  \omega_2\circ s \\
&\ =\  \bar{\omega_2}\circ \truth \omega_2\circ s \\
&\ =\  \bar{\omega_2}\circ t
\mpunct.
\end{align*}
Since
$\bar{\omega_1}\circ t=\bar{\omega_2}\circ t$
is nonzero,
$\bar{\omega_1}=\bar{\omega_2}$
by the uniqueness of normalization.
Therefore $\omega_1=\bar{\omega_1}\circ\truth \omega_1=\bar{\omega_2}\circ \truth \omega_2=\omega_2$.
\end{proof}

\section{Proofs in Section~\ref{sec:classification}}
\label{app:convex-proofs}

To prove Proposition~\ref{prop:sBOUS-equiv-sEMod},
first we establish the connection between
monotone $\sigma$-complete
ordered vector spaces
with order unit
and $\omega$-complete effect modules.

\begin{lemma}
\label{lem:emod-bigvee-commute}
Let $E$ be an $\omega$-complete effect $[0,1]$-module.
For each ascending sequence $(a_n)_{n\in\N}$
in $E$ and $N\in\N$, we have
$\bigvee_n 2^{-N}\cdot a_n = 2^{-N}\cdot \bigvee_n a_n$.
\end{lemma}
\begin{proof}
It suffices to prove
$\bigvee_n (1/2)\cdot a_n = (1/2)\cdot \bigvee_n a_n$,
which implies the claim by induction.
To simplify notation, we write $h=1/2$.
Let $b_n=h\cdot a_n$.
As $\bigvee_n b_n\le h\cdot 1$,
the sum $(\bigvee_n b_n)\ovee (\bigvee_n b_n)$ is defined.
We claim that
$(\bigvee_n b_n)\ovee (\bigvee_n b_n)=\bigvee_n a_n$.
Indeed, $a_n=b_n\ovee b_n\le (\bigvee_n b_n)\ovee (\bigvee_n b_n)$.
If $a_n\le c$,
then $b_n=h\cdot a_n\le h\cdot c$
and hence $\bigvee_n b_n\le h\cdot c$.
Thus $c=h\cdot c\ovee h\cdot c
\ge (\bigvee_n b_n)\ovee (\bigvee_n b_n)$.
Therefore
\[
h\cdot \bigvee_n a_n
\ =\  h\cdot
\qty\Big(\qty\Big(\bigvee_n b_n)\ovee \qty\Big(\bigvee_n b_n))
\ =\  h\cdot \qty\Big(\bigvee_n b_n)\ovee h\cdot \qty\Big(\bigvee_n b_n)
\ =\  \bigvee_n b_n
\ =\  \bigvee_n h\cdot a_n
\mpunct.
\qedhere
\]
\end{proof}

\begin{lemma}
\label{lem:monotone-sigma-complete-iff-omega-complete}
An ordered vector space $A$ with order unit $u$
is monotone $\sigma$-complete
if and only if the unit interval $[0,u]_A$ is
$\omega$-complete.
\end{lemma}
\begin{proof}
The `only if' direction is straightforward.
Conversely,
suppose that $[0,u]_A$ is
$\omega$-complete.
Let $(a_n)_n$ be an ascending sequence in $A$ bounded above.
Let $a'_n=a_n-a_0$,
so that
$(a'_n)_n$ is a positive ascending sequence bounded above.
We can find
$N\in\N$ such that $(a'_n)_n$ is bounded by $2^N u$.
Then $(2^{-N}\cdot a'_n)_n$ is an ascending sequence in $[0,u]_A$,
so there is a supremum $\bigvee_n 2^{-N}\cdot a'_n$ in $[0,u]_A$.
We will show that
$2^{N}\cdot \bigvee_n 2^{-N}\cdot a'_n$ is a supremum of $(a'_n)_n$
in $A$. Clearly
$a'_n\le 2^{N}\cdot \bigvee_n 2^{-N}\cdot a'_n$ for each $n\in\N$.
Suppose that $a'_n\le b$ for each $n\in\N$.
Then we can find $M\in\N$ such that
$b\le 2^M u$ and $N\le M$.
Then we have $\bigvee_n 2^{-M}\cdot a'_n\le 2^{-M}\cdot b$,
and hence
\[
b \ \ge\  2^{M}\cdot \bigvee_n 2^{-M}\cdot a'_n
\ =\  2^{M}\cdot \bigvee_n 2^{-(M-N)}\cdot 2^{-N}\cdot a'_n
\ \overset{\star}= \ 
2^{M}\cdot 2^{-(M-N)}\cdot \bigvee_n 2^{-N}\cdot a'_n
\ =\  2^{N}\cdot \bigvee_n 2^{-N}\cdot a'_n
\mpunct.
\]
Here all $\bigvee$ denote suprema in $[0,u]_A$,
and the equality $\overset{\star}=$ holds
by Lemma~\ref{lem:emod-bigvee-commute}.
Therefore
$(a'_n)_n$ has a supremum in $A$.
It follows that $(a_n)_n=(a_0+a'_n)_n$
has a supremum in $A$ too.
\end{proof}

The following equivalence for morphisms
can be proved similarly by translation and scaling.

\begin{lemma}
\label{lem:monotone-sigma-complete-iff-omega-complete-morphism}
Let $f\colon A\to B$ be a subunital positive linear map
between monotone $\sigma$-complete
ordered vector spaces with order unit.
Then $f$ is $\sigma$-normal if and only if
the restriction $f\colon [0,u]_A\to[0,u]_B$
is $\omega$-continuous.
\qed
\end{lemma}
\begin{Auxproof}
If $f$ is $\sigma$-normal,
then the restriction $f\colon [0,u]_A\to[0,u]_B$
is $\sigma$-continuous,
because the suprema of $\sigma$-chains
in $[0,u]_A$ and $[0,u]_B$
coincide with suprema in $A$ and $B$, respectively.
Conversely, suppose that $f\colon [0,u]_A\to[0,u]_B$
is $\sigma$-continuous.
Let $(a_n)_n$ be an ascending sequence in $A$ bounded above.
Then there exists $r>0$ such that
$(ra_n-ra_0)_n$ is an ascending sequence in $[0,u]_A$.
Because
suprema of $\sigma$-chains in $[0,u]_A$ are suprema in $A$,
and translation and scaling preserve suprema in $A$,
\begin{align*}
r f \qty\Big(\bigvee_n a_n)-r f(a_0)
&= f\qty\Big(r\qty\Big(\bigvee_n a_n)-r a_0) \\
&= f\qty\Big(\bigvee_n (ra_n-ra_0)) \\
&=\bigvee_n f( ra_n-ra_0) \\
&=\bigvee_n (r f(a_n) - rf(a_0)) \\
&= r (\bigvee_n f(a_n)) - rf(a_0)
\end{align*}
Therefore $f (\bigvee_n a_n)=\bigvee_n f(a_n)$.
\end{Auxproof}

In order to prove Proposition~\ref{prop:sBOUS-equiv-sEMod}
we will need the following lemmas.

\begin{lemma}[{\cite[Lemmas~1.1 and 1.2]{Wright1972}}]
\label{lem:monotone-sigma-complete-bous}
Every monotone $\sigma$-complete
ordered vector space with order unit
is a Banach order-unit space.
\qed
\end{lemma}

\begin{lemma}
\label{lem:omega-complete-emod-is-sigma-emod}
Every $\omega$-complete effect $[0,1]$-module
is a $\sigma$-effect $[0,1]$-module.
\end{lemma}
\begin{proof}
Let $E$ be an $\omega$-complete effect $[0,1]$-module.
We need to prove that
the $[0,1]$-action $\cdot\colon [0,1]\times E\to E$
is $\sigma$-biadditive.
By Lemma~\ref{lem:sigma-additive-equiv-omega-conti},
it suffices to prove $\omega$-continuity in each argument.
By Proposition~\ref{prop:OVSu-equiv-EMod}
and Lemmas~\ref{lem:monotone-sigma-complete-iff-omega-complete}
and~\ref{lem:monotone-sigma-complete-bous},
we may assume that $E=[0,u]_A$ for
some monotone $\sigma$-complete Banach order-unit space~$(A,u)$.

\textit{$\omega$-continuity in the first argument:}
Fix $a\in [0,u]_{A}$. We will prove that
$(-)\smul a\colon [0,1]\to [0,u]_{A}$ is $\omega$-continuous.
Let $(r_n)_{n\in\N}$ be an ascending sequence in $[0,1]$.
Clearly $(\bigvee_n r_n)\cdot a$ is an upper bound
of $r_n\cdot a$.
Let $b\in[0,u]_{A}$ satisfy
$r_n\cdot a\le b$ for all $n\in\N$.
Let $N\in\N$ be an arbitrary nonzero number.
Then there is some $m\in\N$ such that
$\bigvee_n r_n< r_m + \frac{1}{N}$,
so that
\[
\qty\Big(\bigvee\nolimits_n r_n)\cdot a
\ \le\  (r_m + \frac{1}{N})\cdot a
\ = \  r_m\cdot a + \frac{a}{N}
\ \le\  b + \frac{u}{N}
\mpunct.
\]
Thus $N\cdot ((\bigvee\nolimits_n r_n)\cdot a - b)\le u$.
Because $N$ is arbitrary and $A$ is Archimedean,
we obtain $(\bigvee\nolimits_n r_n)\cdot a - b\le 0$,
that is, $(\bigvee\nolimits_n r_n)\cdot a\le b$.
Therefore $(\bigvee\nolimits_n r_n)\cdot a=
\bigvee_n (r_n\cdot a)$.

\textit{$\omega$-continuity in the second argument:}
If $r=0$, then $0\cdot (-)\colon [0,u]_{A}\to[0,u]_{A}$
is trivially $\omega$-continuous.
Fix $r\in (0,1]$.
Then $r\cdot (-)\colon A\to A$
is an order isomorphism, with
the monotone inverse $r^{-1}\cdot (-)\colon A\to A$.
Thus $r\cdot (-)\colon A\to A$
preserves all suprema in $A$,
and the restriction
$r\cdot (-)\colon [0,u]_{A}\to [0,u]_{A}$
is $\omega$-continuous.
\end{proof}

\begin{proof}[Proof of Proposition~\ref{prop:sBOUS-equiv-sEMod}]
By Lemmas~\ref{lem:monotone-sigma-complete-iff-omega-complete}
and~\ref{lem:monotone-sigma-complete-iff-omega-complete-morphism},
the equivalence $\OVSu\simeq\EMod[[0,1]]$
of Proposition~\ref{prop:OVSu-equiv-EMod}
restricts to the category
of monotone $\sigma$-complete ordered vector spaces
with order unit and $\sigma$-normal subunital positive linear maps,
and the category of
$\omega$-complete effect $[0,1]$-modules
and $\omega$-continuous additive maps.
These two categories are respectively equal to $\sBOUS$
and $\sEMod[[0,1]]$
by Lemmas~\ref{lem:monotone-sigma-complete-bous}
and~\ref{lem:omega-complete-emod-is-sigma-emod}.
\end{proof}



\begin{proof}[Proof of Proposition~\ref{prop:OVSt-equiv-CWMod}]
The construction of the `inverse' functor $\CWMod[[0,1]]\to \OVSt$
is very much the same as that of
$\EMod[[0,1]]\to \OVSu$ given in \cite[\S\,3.1]{JacobsMF2016}.
We sketch the construction below, and
refer to \cite[\S\,7.2.1]{Cho2019PhD} for further details.

Let $X$ be a cancellative weight $[0,1]$-module.
The totalization \cite{jacobs2012coreflections}
of the PCM $X$
is the commutative monoid $\Tz(X)=\Mlt(X)/{\sim}$
where $\Mlt(X)$
is the free commutative monoid on $X$
consisting of finite multisets on $X$,
denoted as formal finite sums $\sum_i n_i \cdot x_i$
for $n_i\in\N$ and $x_i\in X$,
and $\sim$
is the smallest monoid congruence
such that $1\cdot (x\ovee y)\sim 1\cdot x + 1\cdot y$
and $1\cdot 0\sim 0$.
There is an embedding $X\to \Tz(X)$
given by $x\mapsto 1\cdot x$ which is injective.
Because $X$ is a weight $[0,1]$-module,
$\Tz(X)$ can be equipped with
an monoid action $\Rpos\times \Tz(X)\to \Tz(X)$,
and the weight map
extends to $\weight{-}\colon \Tz(X)\to \Rpos$.
By cancellativity of $X$,
we can prove that $\Tz(X)$ is a cancellative monoid.

We then define $V(X)=(\Tz(X)\times \Tz(X))/{\approx}$
where $\approx$ is defined by
$(a,b)\approx (c,d)$ iff $a+d=b+c$.
Because $\Tz(X)$ is cancellative,
$\Tz(X)$ embeds into the Abelian group $V(X)$
by $a\mapsto (a,0)$.
Now $V(X)$ forms a real vector space with
the scalar multiplication $r(a,b)=(ra,rb)$
for $r\ge 0$
and $r(a,b)=((-r)b,(-r)a)$
for $r< 0$.
With $\Tz(X)$ embedded in $V(X)$
as a positive cone, $V(X)$ forms an ordered vector space.
Moreover, $V(X)$ is positively generated and equipped with
trace $\tau\colon V(X)\to \R$ given by $\tau(a,b)=\weight{a}-\weight{b}$.
\end{proof}

The following lemma is similar to
\cite[Proposition~2.4.11 and Lemma 2.4.12]{Furber2017}
and \cite[Corollary~2]{BoergerK1996} (see also \cite{EdwardsG1970}),
but here stated in terms of weight modules
instead of convex sets.

\begin{lemma}
\label{lem:bbns-if-subbase-sigma-wmod}
Let $V$ be an ordered vector space with trace $\tau$.
Assume that the subbase $\sBase(V)$
forms a $\sigma$-weight $[0,1]$-module,
extending its canonical weight $[0,1]$-module structure.
Then $V$ is a Banach pre-base-norm space.
Moreover, for each countable summable family $(x_n)_{n\in\N}$ in
the $\sigma$-weight $[0,1]$-module $\sBase(V)$,
the series $\sum_{n=0}^{\infty} x_n$ converges to
$\bigovee_{n\in\N} x_n$ with respect to the base norm.
\end{lemma}
\begin{proof}
We first prove that $V$ is a pre-base-norm space (i.e.\ that the seminorm is actually a norm).
Let $a\in V$ satisfy $\norm{a}=0$.
Let $\tilde{x},\tilde{y}\in V_+$ be such that $a=\tilde{x}-\tilde{y}$.
Let $r= \max(\tau(\tilde{x}),\tau(\tilde{y}))$.
If $r=0$, we have $a=0$.
Otherwise,
writing $x=r^{-1}\tilde{x}$ and $y=r^{-1}\tilde{y}$,
we have $x,y\in\sBase(V)$ and $r\norm{x-y}=\norm{a}=0$,
so that $\norm{x-y}=0$.
It suffices to prove that $x-y=0$.

By $\norm{x-y}=0$, for each $n\in\N$
we can find $z_n,w_n\in V_+$ such that
$x-y=w_n-z_n$
and $\tau(z_n)+\tau(w_n)\le 1/2^{n+1}$.
Note that $z_n,w_n\in \sBase(V)$
and by $w_n-z_n=x-y=w_{n+1}-z_{n+1}$,
we have $z_n + w_{n+1}=z_{n+1} + w_n$.
Because $\sum_{n\in\N}\weight{z_n}+\sum_{n\in\N}\weight{w_n}\le 1$,
the following countable sums exist
in the $\sigma$-weight $[0,1]$-module $\sBase(V)$,
and the equations hold by partition-associativity.
\begin{align*}
z_0\ovee
\qty\Big(\bigovee_{n=1}^{\infty} z_n)
\ovee \qty\Big(\bigovee_{n=1}^{\infty} w_n)
&\ =\ 
\bigovee_{n=0}^{\infty}
(z_n\ovee w_{n+1})
\\
&\ =\ 
\bigovee_{n=0}^{\infty}
(z_{n+1}\ovee w_{n})
\ =\ 
w_0\ovee
\qty\Big(\bigovee_{n=1}^{\infty} z_n)
\ovee \qty\Big(\bigovee_{n=1}^{\infty} w_n)
\end{align*}
By cancellation, $z_0=w_0$, so that $x-y=w_0-z_0=0$.

Before proving that $V$ is a Banach space,
we prove the claim about convergence.
Let $(x_n)_n$ be
a countable family summable in $\sBase(V)$.
Using the fact that $\weight{x}\equiv\tau(x)=\norm{x}$
for $x\in \sBase(V)$ --- see
\cite[Corollary 2.2.5]{Furber2017} ---
we have for each $N\in\N$
\[
\norm\Big{
\qty\Big(\bigovee_{n\in\N} x_n)
-
\qty\Big(\sum_{n=0}^N x_n)}
\ =\ 
\norm\Big{
\bigovee_{n=N+1}^{\infty} x_n}
\ =\ 
\weight[\Big]{
\bigovee_{n=N+1}^{\infty} x_n}
\ =\ 
\sum_{n=N+1}^{\infty} \weight{x_n}
\mpunct.
\]
Because
$\lim_{N\to\infty}\sum_{n=0}^{N} \weight{x_n}= \sum_{n=0}^{\infty} \weight{x_n}$
and $\sum_{n=0}^{N} \weight{x_n}+
\sum_{n=N+1}^{\infty} \weight{x_n}=\sum_{n=0}^{\infty} \weight{x_n}<\infty$
we must have\linebreak
$\lim_{N\to\infty}\sum_{n=N+1}^{\infty} \weight{x_n}= 0$.
Therefore the series $\sum_{n=0}^{\infty} x_n$ converges to $\bigovee_{n\in\N} x_n$.

Finally we prove that $V$ is a Banach space.
It suffices to prove that
every absolutely convergent series converges.
Let $(x_n)_{n\in\N}$
be an absolutely convergent series.
Without loss of generality we may assume that
$\sum_{n=0}^{\infty} \norm{x_n}\le 1/2$
and $\norm{x_n}\neq 0$ for all $n\in\N$.
For each $n\in\N$ we can find $y_n,z_n\in V_+$
such that $\tau(y_n)+\tau(z_n)< 2\norm{x_n}$
and $x_n=y_n-z_n$.
Because $\tau(y_n)+\tau(z_n)< 2\norm{x_n}\le 1$,
we have $y_n,z_n\in \sBase(V)$.
Moreover we have
\[
\sum_{n=0}^\infty\weight{y_n}
\ =\ \sum_{n=0}^\infty\tau(y_n)
\ \le\  \sum_{n=0}^\infty 2\norm{x_n}
\ \le\  1
\]
and similarly $\sum_{n=0}^\infty\weight{z_n}\le 1$,
that is, $(y_n)_n$ and $(z_n)_n$ are summable in $\sBase(V)$.
Let $a= \bigovee_n y_n$
and $b= \bigovee_n z_n$.
By what we have shown above,
$\sum_{n=0}^N y_n\to a$ and $\sum_{n=0}^N z_n\to b$
when $N\to \infty$.
Therefore $\sum_{n=0}^N x_n=(\sum_{n=0}^N y_n)-
(\sum_{n=0}^N z_n)\to a-b$ when $N\to \infty$.
\end{proof}

\begin{proof}[Proof of Proposition~\ref{prop:sBBNS-equiv-sCWMod}]
It is easy to see that for each $V\in \sBBNS$,
the subbase $\sBase(V)$ forms a $\sigma$-weight $[0,1]$-module
whose countable addition is given by sums of series.
By this fact and
Lemma~\ref{lem:bbns-if-subbase-sigma-wmod},
the equivalence $\OVSt\simeq\CWMod[[0,1]]$
can be restricted to
$\sBBNS$ and the full subcategory
of $\CWMod[[0,1]]$ consisting of cancellative weight $[0,1]$-modules
that have an extension to a $\sigma$-weight $[0,1]$-module.
Let $\CWMod[[0,1]]'$ denote this subcategory.
There is a bijection between objects
of $\CWMod[[0,1]]'$ and $\sCWMod[[0,1]]$,
because an extension of a weight $[0,1]$-module to
a $\sigma$-weight $[0,1]$-module is unique by
Lemma~\ref{lem:bbns-if-subbase-sigma-wmod}.
Let $f\colon X\to Y$ be a morphism in $\CWMod[[0,1]]'$.
Then we can represent
$X$ and $Y$ respectively as
$\sBase(V_X)$ and $\sBase(V_Y)$ for some $V_X,V_Y\in \sBBNS$,
and $f$ extends to a morphism $V_X\to V_Y$ in $\sBBNS$.
Because
the countable sums in $\sBase(V_X),\sBase(V_Y)$ are given by convergent series
and $f$ is continuous,
$f$ preserves countable sums,
i.e.\ it is a morphism in $\sCWMod[[0,1]]$.
We conclude that
$\CWMod[[0,1]]'$ is isomorphic to $\sCWMod[[0,1]]$.
\end{proof}

\end{document}